\documentclass[3p]{elsarticle}

\usepackage{amsmath}
\usepackage{amsthm}
\usepackage{amsfonts}
\usepackage{amssymb}

\usepackage{lineno,hyperref}
\modulolinenumbers[5]

\newtheorem{theorem}{Theorem}[section]
\newtheorem{lemma}[theorem]{Lemma}
\newtheorem{prop}[theorem]{Proposition}
\newtheorem{cor}[theorem]{Corollary}

\theoremstyle{definition}
\newtheorem{rem}[theorem]{Remark}
\theoremstyle{definition}
\newtheorem{definition}[theorem]{Definition}

\newtheorem{hypothesis}[theorem]{Hypothesis}

\DeclareMathOperator{\dv}{div}

\newcommand{\pO}{{\partial \Omega}}
\newcommand{\pOt}{{\partial \Omega_t}}
\newcommand{\Ht}{H^{3/2}(\pO)}
\newcommand{\Hp}{H^{1/2}(\pO)}
\newcommand{\Hm}{H^{-1/2}(\pO)}
\newcommand{\Hh}{\Hp \oplus \Hm}

\newcommand{\wtf}{\widetilde f}
\newcommand{\wtg}{\widetilde g}
\newcommand{\wth}{\widetilde h}
\newcommand{\wtX}{\widetilde X}
\newcommand{\wtp}{\widetilde \varphi}

\newcommand{\p}{\partial}

\newcommand{\bbR}{{\mathbb{R}}}

\newcommand{\cD}{{\mathcal D}}

\newcommand{\cH}{{\mathcal H}}

\newcommand{\cO}{{\mathcal O}}

\DeclareMathOperator{\Tr}{Tr}
\DeclareMathOperator{\tr}{tr}

\newcommand{\loc}{\text{\rm{loc}}}

\newcommand{\beq}{\begin{equation}}
\newcommand{\enq}{\end{equation}}

\journal{Journal of \LaTeX\ Templates}

\begin{document}

\begin{frontmatter}

\title{A dynamical approach to semilinear elliptic equations}

\author[MBaddress]{Margaret Beck}
\ead{mabeck@bu.edu}
\address[MBaddress]{Department of Mathematics and Statistics, Boston University, Boston, MA 02215, USA}

\author[GCaddress]{Graham Cox\corref{mycorrespondingauthor}}
\cortext[mycorrespondingauthor]{Corresponding author}
\ead{gcox@mun.ca}
\address[GCaddress]{Department of Mathematics and Statistics, Memorial University of Newfoundland, St. John's, NL A1C 5S7, Canada}

\author[CJaddress]{Christopher Jones}
\ead{ckrtj@email.unc.edu}
\address[CJaddress]{Department of Mathematics, University of North Carolina at Chapel Hill, Chapel Hill, NC 27599, USA}

\author[YLaddress]{Yuri Latushkin}
\ead{latushkiny@missouri.edu}
\address[YLaddress]{Department of Mathematics, University of Missouri, Columbia, MO 65211, USA}

\author[ASaddress]{Alim Sukhtayev}
\ead{sukhtaa@miamioh.edu}
\address[ASaddress]{Department of Mathematics, Miami University, Oxford, OH 45056, USA}

\begin{abstract}
A characterization of a semilinear elliptic partial differential equation (PDE) on a bounded domain in $\mathbb R^n$ is given in terms of an infinite-dimensional dynamical system. The dynamical system is on the space of boundary data for the PDE. This is a novel approach to elliptic problems that enables the use of dynamical systems tools in studying the corresponding PDE. The dynamical system is ill-posed, meaning solutions do not exist forwards or backwards in time for generic initial data. We offer a framework in which this ill-posed system can be analyzed. This can be viewed as generalizing the theory of spatial dynamics, which applies to the case of an infinite cylindrical domain.
\end{abstract}

\begin{keyword}
semilinear equations \sep spatial dynamics \sep dynamical systems
\MSC[2010] 35J67 \sep  35A24 \sep 34D09 \sep 35J25
\end{keyword}

\end{frontmatter}


\section{Introduction}
A standard trick in dynamical systems is to write the differential equation $u_{xx} + F(u) = 0$ as a first-order system
\begin{align*}
	u_x &= v \\
	v_x &= -F(u).
\end{align*}
This allows for the application of dynamical systems methods, such as phase plane analysis, exponential dichotomies, and the Evans function; see, for instance, \cite{KPP13} and references therein for a modern perspective.

Similarly, on an infinite cylindrical domain $\Omega = \bbR \times \Omega' \subset \bbR^n$, the semilinear partial differential equation $\Delta u + F(u) = 0$ can be written in the form
\begin{align}\label{eqn:cylinder}
\begin{split}
	u_x &= v \\
	v_x &= -F(u) - \Delta_y u,
	\end{split}
\end{align}
where $(x,y) \in \bbR \times \Omega'$ and $\Delta_y$ denotes the Laplacian on the cross-section $\Omega' \subset \bbR^{n-1}$. In this case the phase space is infinite-dimensional, and the analysis requires more care. In particular, the equation is ill-posed both forwards and backwards in time. As a result, it is nontrivial to prove existence of solutions. The idea of rewriting the semilinear PDE as an evolution equation along the cylindrical direction is the basis of the area now known as {\em Spatial Dynamics}, see \cite{A84,BSZ10,DSSS09,G86,K82,LP08,M86,M88,M91,PSS97,S02,SS01,S03}.

In this paper we extend this correspondence to general Euclidean domains. That is, we obtain the analog of \eqref{eqn:cylinder} for a bounded domain $\Omega$ which is smoothly contracted to a point through a one-parameter family $\{\Omega_t\}$. In this case $t$ becomes the dynamical variable with respect to which we study the evolution of the boundary data. While similar in spirit to the cylindrical case described above, the analysis is complicated by the nontrivial geometry and the fact that the resulting system of equations becomes singular as the domain degenerates to a point.

\subsection*{Outline of the paper}
In Section \ref{sec:harmonic} we motivate our general construction and results by studying harmonic functions in $\bbR^3$, where the computations can be done explicitly. In Section \ref{sec:results} we present the general framework and state all of the major results. Section \ref{sec:geo} contains some geometric preliminaries that will be needed for our analysis. The infinite-dimensional dynamical system is studied in Section \ref{sec:evolution}, where we prove its equivalence to the original semilinear PDE.  Finally, in Section \ref{sec:dichotomy} we describe exponential dichotomies for the linearized dynamical system, in particular proving that the unstable dichotomy subspace (if it exists) coincides with the space of Cauchy data of weak solutions to the PDE.

\subsection*{Acknowledgements}
The authors would like to acknowledge the support of the American Institute of Mathematics and the Banff International Research Station, where much of this work was carried out. M.B. acknowledges the support of NSF grant DMS-1411460 and of an AMS Birman Fellowship. G.C. acknowledges the support of NSERC grant RGPIN-2017-04259. C.J. was supported by ONR grant N00014-18-1-2204. Y.L. was supported by the NSF grant  DMS-1710989, by the Research Board and Research Council of the University of Missouri, and by the Simons Foundation. A.S. was supported by NSF grant DMS-1910820.

\section{A motivating example: harmonic functions in $\bbR^3$}\label{sec:harmonic}

Suppose that $u(r,\theta,\phi)$ is a harmonic function in $\bbR^3$. Let $\Omega_t = \{x : |x| < t \} \subset \bbR^3$, and consider the functions
\begin{align*}
	f(t) := u(t,\cdot,\cdot), \quad
	g(t) := \frac{\p u}{\p r}(t,\cdot,\cdot),
\end{align*}
which are in $C^\infty(S^2)$ for $t>0$. We refer to the pair $(f(t),g(t))$ as the \emph{Cauchy data} (or \emph{boundary data}) of $u$ on the surface $\pOt = \{x \in \bbR^3 : |x| = t\}$. Note that $f(t)$ is just the function $u$ evaluated at radius $r=t$. We have introduced the new variable $t$ to emphasize that we are viewing this as an evolutionary variable, rather than a spatial coordinate.

Differentiating $f$ with respect to $t$, we obtain
\[
	\frac{df}{dt} = \frac{\p u}{\p r} \bigg|_{r=t} = g(t).
\]
To differentiate $g$ we use the formula $\Delta u = u_{rr} + 2 r^{-1} u_r + r^{-2} \Delta_{S^2} u$, where $\Delta_{S^2}$ denotes the Laplace--Beltrami operator on the sphere:
\[
	\Delta_{S^2} f = \frac{1}{\sin\theta} \frac{\p}{\p\theta} \left( \sin\theta \frac{\p f}{\p \theta} \right) + \frac{1}{\sin^2\theta} \frac{\p^2 f}{\p \phi^2}.
\]
Since $\Delta u = 0$, it follows that
\begin{align*}
	\frac{dg}{dt} &= \frac{\p^2 u}{\p r^2} \bigg|_{r=t} = - \left.\left( \frac{2}{r} \frac{\p u}{\p r} + \frac{1}{r^2} \Delta_{S^2} u \right) \right|_{r=t} = - \frac{1}{t^2} \Delta_{S^2} f(t) - \frac{2}{t} g(t)
\end{align*}
and so for all $t>0$, $f$ and $g$ satisfy the linear system
\begin{align}\label{R3FGham}
	\frac{d}{dt} \begin{pmatrix} f \\ g \end{pmatrix} = 
	\begin{pmatrix} 0 & 1 \\ - t^{-2} \Delta_{S^2} & -2 t^{-1} \end{pmatrix} \begin{pmatrix} f \\ g \end{pmatrix}.
\end{align}

The operator appearing on the right-hand side of \eqref{R3FGham} has spectrum unbounded in both directions. As a result, the system is ill-posed, meaning one cannot expect a solution to exist forward \emph{or} backward in time for generic initial data. However, this system admits an \emph{exponential dichotomy}---that is, a splitting of the phase space into two subspaces, both infinite-dimensional, on which solutions exists forward and backward in time, respectively.

To see this, we first rescale $f$ and $g$, multiplying them by appropriate powers of $t$, namely $t^{\alpha} f(t)$ and $t^{1+\alpha} g(t)$, where $\alpha$ is a real constant to be determined. We then reparametrize by defining a new variable $\tau = \log t$, resulting in the functions
\[
	\widetilde f(\tau) = e^{\alpha\tau} f(e^\tau), \quad \widetilde g(\tau) = e^{(1+\alpha)\tau} g(e^\tau),
\]
which are defined for all $\tau \in \bbR$. It follows from \eqref{R3FGham} that
\begin{align}\label{R3FG}
	\frac{d}{d\tau} \begin{pmatrix} \widetilde f \\ \widetilde g \end{pmatrix} = 
	\begin{pmatrix} \alpha & 1\\ - \Delta_{S^2} & \alpha - 1 \end{pmatrix} \begin{pmatrix} \widetilde f \\ \widetilde g \end{pmatrix}
\end{align}
for all $\tau \in \bbR$. The eigenvalues of the operator matrix on the right-hand side are
\[
	\nu_l^\pm = \frac{(2\alpha - 1) \pm\sqrt{4\mu_l + 1}}{2},
\]
where $0 = \mu_0 < \mu_1 < \mu_2 < \cdots$ are the distinct eigenvalues of $-\Delta_{S^2}$. These are given by $\mu_l = l(l+1)$ for integers $l \geq 0$. The corresponding eigenfunctions are the spherical harmonics $Y_l^m(\theta,\phi)$ for $-l \leq m \leq l$, hence $\mu_l$ has multiplicity $2l+1$. It follows that
\[
	\nu^+_l = \alpha + l, \quad \nu^-_l = \alpha - l - 1,
\]
with the corresponding solutions to \eqref{R3FG} given by
\begin{align*}
	\Big( \wtf^+_{lm}(\tau), \wtg^+_{lm}(\tau)\Big) &= \big( e^{(\alpha + l) \tau } Y_l^m, \, l e^{(\alpha + l) \tau } Y_l^m \big) \\
	\Big( \wtf^-_{lm}(\tau), \wtg^-_{lm}(\tau)\Big) &= \big( e^{(\alpha - l - 1) \tau} Y_l^m, \, -(l+1) e^{(\alpha - l - 1) \tau} Y_l^m \big)
\end{align*}
for $-l \leq m \leq l$. Undoing the scaling and reparametrization yields
\begin{align*}
	\big( f^+_{lm}(t), g^+_{lm}(t) \big) &= \big(t^l Y_l^m, \, l t^{l-1} Y_l^m \big)  \\
	\big( f^-_{lm}(t), g^-_{lm}(t) \big) &= \big( t^{-l-1} Y_l^m, \, -(l+1) t^{-l-2} Y_l^m \big).
\end{align*}
Observe that the pair $\big(f^+_{lm}(t), g^+_{lm}(t) \big)$ is precisely the Cauchy data on $\pOt$ of the harmonic function $u(r,\theta,\phi) = r^l Y_l^m(\theta,\phi)$. Similarly, $\big(f^-_{lm}(t), g^-_{lm}(t) \big)$ is the Cauchy data of $u(r,\theta,\phi) = r^{-l-1} Y_l^m(\theta,\phi)$.

For any $\tau_0 \in \bbR$, the \emph{unstable subspace} of \eqref{R3FG}, denoted $E^u(\tau_0)$, consists of functions $(\widetilde f_0, \widetilde g_0)$ with the property that there exists a solution $(\widetilde f(\tau),\widetilde g(\tau))$ to \eqref{R3FG} that is defined for all $\tau \leq \tau_0$, satisfies the terminial condition $(\widetilde f(\tau_0),\widetilde g(\tau_0)) = (\wtf_0, \wtg_0)$, and decays exponentially as $\tau\to-\infty$. Similarly, the \emph{stable subspace} of \eqref{R3FG}, $E^s(\tau_0)$, consists of functions $(\widetilde f_0, \widetilde g_0)$ with the property that there exists a solution $(\widetilde f(\tau),\widetilde g(\tau))$ to \eqref{R3FG} that is defined for all $\tau \geq \tau_0$, satisfies the initial condition $(\widetilde f(\tau_0),\widetilde g(\tau_0)) = (\wtf_0, \wtg_0)$, and decays exponentially as $\tau\to\infty$. To determine the stable and unstable subspaces, we must identify the solutions $(\widetilde f^\pm_{lm},\widetilde g^\pm_{lm})$ for which the corresponding spatial eigenvalue $\nu^\pm_l$ is negative, and those for which it is positive, respectively. This depends on the scaling parameter $\alpha$, which has not yet been specified. We seek $\alpha$ so that the unstable subspace corresponds to the Cauchy data of all harmonic functions that are bounded at the origin. This will be the case if $\nu^+_l > 0$ and $\nu^-_l < 0$ for all $l$. This is equivalent to $\nu^-_0 < 0 < \nu^+_0$, and so any $\alpha \in (0,1)$ will suffice.

In summary, we have seen that for $0 < \alpha < 1$ the system \eqref{R3FG} admits an exponential dichotomy such that: 1) the unstable subspace $E^u(\tau)$ consists of the Cauchy data on the surface $\{r = e^\tau\}$ of harmonic functions that are bounded at the origin; and 2) the stable subspace $E^s(\tau)$ consists of the Cauchy data on $\{r = e^\tau\}$ of harmonic functions that decay at infinity.

\begin{rem}\label{rem:2Dharmonic}
A similar analysis carries through in the planar case, and an exponential dichotomy arises in the same manner. However, the situation is complicated by the fact that the evolution equation has a two-dimensional center subspace, corresponding to the harmonic functions $1$ and $\log r$. While $\log r$ blows up as $r\to0$, it does so very slowly, in the sense that $r^\alpha \log r \to 0$ for any $\alpha > 0$. On the other hand, if $\alpha < 0$, then both $r^\alpha$ and $r^\alpha \log r$ are unbounded at the origin. As a result, no choice of $\alpha$ is able to distinguish (in terms of growth or decay) $\log r$ from a constant function. Therefore, the stable and unstable subspaces do not admit the same interpretation as in the higher-dimensional case. This phenomenon will be observed again below; see Corollary \ref{cor:alphagap} and Remark \ref{rem:n2},
\end{rem}

The main objective of this paper is to generalize the preceding constructions to semilinear elliptic equations on $\bbR^n$.

\section{Definitions and results}\label{sec:results}
We generalize \eqref{eqn:cylinder} by considering a smooth family of domains $\{\Omega_t\}$ in $\bbR^n$ and describing the time evolution of the quantities
$\left.u\right|_{\pOt}$ and $\p u/\p \nu\big|_{\pOt}$, where $u \colon \bbR^n \to \bbR$ solves the semilinear equation
\begin{align}\label{eqn:PDE}
	\Delta u + F(x,u) = 0.
\end{align}

We first describe the types of domains $\Omega_t$ to which our method applies. We let
\begin{align}\label{omegadef}
	\Omega_t = \{x \in \bbR^n : \psi(x) < t^2 \},
\end{align}
for a suitable function $\psi\colon \bbR^n \to \bbR$. We assume the following for the remainder of the paper.

\begin{hypothesis}\label{hyp:geo}
The function $\psi$ has the following properties:
\begin{enumerate}
	\item $\psi \in C^3(\bbR^n, \bbR)$;
	\item $\psi$ has a nondegenerate minimum at $x=0$, with $\psi(0) = 0$;
	\item $\psi$ has no other critical points;
	\item $\psi$ is proper (i.e. preimages of compact sets are compact).
\end{enumerate}
\end{hypothesis}

These assumptions on $\psi$ are motivated by the example $\psi(x) = |x|^2$, which satisfies Hypothesis \ref{hyp:geo}, and leads to the family of domains $\Omega_t = \{x : |x|<t\}$. In general, the nondegeneracy of $\psi$ ensures the domains shrink to a point in a sufficiently regular manner at $t \to 0$. By the Morse lemma there exist coordinates $(y_1, \ldots, y_n)$ near the origin such that $\Omega_t = \left\{ y_1^2 + \cdots y_n^2 < t^2 \right\}$; see \cite{M63}. In this sense any function $\psi$ satisfying Hypothesis \ref{hyp:geo} locally resembles $|x|^2$.

For any $0 \leq a < b < \infty$ we define
\begin{align}\label{annulus}
	\Omega_b = \{x \in \bbR^n : \psi(x) < b^2 \}, \quad \Omega_{a,b} = \{x \in \bbR^n : a^2 < \psi(x) < b^2\},
\end{align}
so that $\Omega_b$ is diffeomorphic to an open ball and $\Omega_{a,b}$ is diffeomorphic to an annulus. A case of particular interest is $a=0$, where the domain is a punctured ball, $\Omega_{0,b} = \Omega_b \setminus \{0\}$.

To understand the evolution of $u$ and its normal derivative restricted to $\pOt$, we need a smooth parameterization of the domains. For convenience we define a fixed  ``reference domain" $\Omega$ by
\begin{align}
	\Omega = \Omega_1 = \{x \in \bbR^n : \psi(x) < 1\}.
\end{align}
The dynamical system we formulate is defined on the boundary, $\pO = \{x \in \bbR^n : \psi(x) = 1\}$. This is related to each $\pOt$ by a family of diffeomorphisms $\{\varphi_t\}$ whose existence is established in Section~\ref{sec:flow}.

\begin{lemma}\label{lemma:phit}
Suppose $\psi$ satisfies Hypothesis \ref{hyp:geo}, and define $\{\Omega_t\}_{t>0}$ by \eqref{omegadef}. Then there exists a family of diffeomorphisms $\{\varphi_t\}_{t>0}$ on $\bbR^n$ such that $\varphi_t(\Omega) = \Omega_t$ for each $t>0$, and
\[
	\varphi_s \circ \varphi_t = \varphi_{st}
\]
for any $s,t > 0$. In particular, $\varphi_1 = \operatorname{id}$.
\end{lemma}

It follows that $\varphi_s(\Omega_t) = \Omega_{st}$ for any $s,t > 0$. The family $\{\varphi_t\}_{t>0}$ satisfies a group property with respect to the multiplicative group of positive real numbers. Perhaps more naturally, it can be viewed as an additive group with respect to the variable $\tau = \log t$, because $\varphi_{\exp(\tau_1)} \circ \varphi_{\exp(\tau_2)} = \varphi_{\exp(\tau_1 + \tau_2)}$.

The flow $\{\varphi_t\}$ is generated by a nonautonomous vector field $X$, satisfying
\begin{align}\label{Xdef}
	X(\varphi_t(x),t) = \frac{d }{dt} \varphi_t(x)
\end{align}
for any $x \in \bbR^n$ and $t>0$. We define a function $\sigma \colon \bbR^n \setminus\{0\} \to \bbR$ as follows. If $x \neq 0$, then $x \in \pOt$ for some $t>0$, namely $t = t(x) = \sqrt{\psi(x)}$. Using this, we let
\begin{align}\label{sigmadef}
	\sigma(x) = X(x,t(x)) \cdot \nu_x,
\end{align}
where $\nu_x$ denotes the outward unit normal to $\pO_{t(x)}$ at the point $x$. This function can in fact be computed directly from $\psi$; see \eqref{sigmaexplicit}. Next, for each $t>0$ we define a function $\sigma_t \colon \pO \to \bbR$ by
\begin{align}\label{sigmatdef}
	\sigma_t(y) = \sigma(\varphi_t(y)).
\end{align}
This measures the normal speed at which a point $y \in \pO$ moves under the flow, since
\[
	\frac{d}{dt} \varphi_t(y) \cdot \nu_{\varphi_t(y)} = X(\varphi_t(y),t) \cdot \nu_{\varphi_t(y)} = \sigma(\varphi_t(y)) = \sigma_t(y).
\]
At any point $x \in \bbR^n$ and $t>0$ we denote the tangential component of $X(x,t)$ by $\gamma(x,t)$, so we have the decomposition $X(x,t) = \big(X(x,t) \cdot \nu_x \big) \nu_x + \gamma(x,t)$ into normal and tangential components. If $t = t(x)$, this simplifies to
\begin{align}
	X(x,t) = \sigma(x) \nu_x + \gamma(x,t).
\end{align}
In the following sections we will always have $x = \varphi_t(y)$ for some $y \in \pO$, and hence $t = t(x)$.

We next define the Cauchy data of a solution to \eqref{eqn:PDE}.
For $u \in C^1(\bar\Omega)$ we define functions $f \colon (0,\infty) \to C^1(\pO)$ and $g \colon (0,\infty) \to C^0(\pO)$ by
\begin{align}\label{fgdef}
	f(t)(y) = u(\varphi_t(y)), \quad
	g(t)(y) = \frac{\p u}{\p \nu}(\varphi_t(y)), \quad y \in \pO,
\end{align}
then combine these to form the \emph{trace},
\begin{align}\label{def:tr}
	\Tr_t u = \left( f(t), g(t) \right).
\end{align}
Observe that $f(t)$ is just the restriction of $u$ to $\pOt$, pulled back to $\pO$ via the diffeomorphism $\varphi_t$, and similarly for $g(t)$. The advantage of $f$ and $g$ is that their domains are $t$-independent.

Now suppose that $u$ is a solution to \eqref{eqn:PDE}. If $u$ is suitably smooth, one can show (see Section \ref{sec:localproof}) that $f$ and $g$ satisfy the system of equations
\begin{align}\label{eqn:ODE}
\begin{split}
	\frac{df}{dt} &= T_t f + \sigma_t g \\
	\frac{dg}{dt} &= - \sigma_t F_t(f) - L_t f + (T_t - \sigma_t H_t) g,
\end{split}
\end{align}
where $H_t = H_{\pOt} \circ \varphi_t\big|_\pO$, with $H_{\pOt}$ denoting the mean curvature of $\pOt$, and $F_t(f) \colon \pO \to \bbR$ is defined by $F_t(f)(y) = F(\varphi_t(y),f(t)(y))$. Additionally, $T_t$ and $L_t$ are the differential operators
\begin{align}
	T_t f &= \left[ \gamma \cdot \nabla^{\pOt} \left(f \circ \varphi_t^{-1}\right) \right] \circ \varphi_t  
	\label{Ttdef}  \\
	L_t f &= \dv^{\pOt} \left[\sigma \nabla^{\pOt}(f \circ \varphi_t^{-1})\right] \circ \varphi_t.
	\label{Ltdef}
\end{align} 
In \eqref{Ttdef}, $\nabla^{\pOt}$ denotes the tangential part of the gradient, computed as
\begin{align}
	\nabla^{\pOt} u = \nabla u - \frac{\p u}{\p \nu} \nu 
\end{align}
for any function $u$ defined in a neighborhood of $\pOt$. It is easily seen that this only depends on the restriction $u\big|_{\pOt}$. The tangential divergence, $\dv^\pOt$, is minus the formal adjoint of $\nabla^{\pOt}$. For any vector field $Y$ defined in a neighborhood of $\pOt$ we can write
\begin{align}\label{div:nortan}
	(\dv Y )\big|_{\pOt} 
	= \dv^{\pOt} \big(Y^{\pOt} \big) + (Y \cdot \nu) H_{\pOt}  + \nu \cdot \nabla_\nu Y
\end{align}
where $Y^{\pOt} = Y - (Y \cdot \nu)\nu$ is the tangential part of $Y$. In particular, when $Y$ is tangential to $\pOt$, we have $Y \cdot \nu = 0$, hence $\nu \cdot \nabla_\nu Y = -Y \cdot \nabla_\nu \nu$, and so $\dv^\pOt Y = \dv Y + Y \cdot \nabla_\nu \nu$.

To make the notion of a solution to \eqref{eqn:ODE} precise, we define the Hilbert spaces
\begin{align}\label{def:H}
	\cH = \Hp \oplus \Hm, \quad \cH^1 = \Ht \oplus \Hp.
\end{align}

\begin{definition}\label{def:ODEsol}
Let $J \subset \bbR_+ = (0,\infty)$ be an open interval. The pair $(f,g)$ is said to be a solution to \eqref{eqn:ODE} on $J$ if
\[
	(f,g) \in C^0(J, \cH^1) \cap C^1(J,\cH) \cap C^0(\bar{J},\cH), \quad F_{t}(f) \in L_{\loc}^2(J, L^2(\pO)),
\]
and $(f,g)$ satisfies \eqref{eqn:ODE} on $J$ with values in $\cH$. Here $\bar J$ denotes the closure of $J$ in $\bbR_+$, so $\overline{(0,T)} = (0,T]$ for any $T<\infty$.
\end{definition}

We also need to define the notion of a weak solution to the semilinear problem \eqref{eqn:PDE}.

\begin{definition}\label{def:PDEsol}
Let $\Omega \subset \bbR^n$ be a bounded domain with Lipschitz boundary (such as $\Omega_{a,b}$ or $\Omega_b$ for some $0 < a < b < \infty$). A function $u$ is said to be a weak solution to \eqref{eqn:PDE} on $\Omega$ if $u \in H^1(\Omega)$, $F(\cdot,u) \in L^2(\Omega)$, and
\begin{align}\label{def:weak}
	\int_\Omega \nabla u \cdot \nabla v = \int_\Omega F(\cdot,u) v \quad \text{ for all }v \in H^1_0(\Omega).
\end{align}
We then say that $u$ is a weak solution on $\Omega_{0,b}$ if it is a weak solution on $\Omega_{a,b}$ for all $a \in (0,b)$. Finally, $u$ is a weak solution on $\bbR^n$ (resp. $\bbR^n \setminus \{0\}$) if it is a weak solution on $\Omega_b$ (resp. $\Omega_{0,b}$) for all $b > 0$.
\end{definition}

\begin{rem}
More generally, \eqref{def:weak} makes sense for any $F(\cdot,u) \in H^{-1}(\Omega)$. For instance, this will be the case if $F$ satisfies a uniform growth assumption $|F(x,z)| \leq C |z|^{(n+2)/(n-2)}$ for all $x \in \Omega$ and $z \in \bbR$. However, the stronger condition $F(\cdot,u) \in L^2(\Omega)$ is needed in the proof of Theorem \ref{thm:local} to ensure that $u \in H^2_{\rm loc}(\Omega_{a,b})$, and hence $\Tr_t u = (f(t),g(t)) \in \cH^1$ for $a < t < b$.
\end{rem}

We can now state our first result relating the boundary data $(f,g)$ to $u$. It says that the PDE \eqref{eqn:PDE} on the deleted ball $\Omega_{0,T} = \Omega_T \setminus \{0\}$ is equivalent to the ODE \eqref{eqn:ODE} on the interval $(0,T)$.

\begin{theorem}\label{thm:local}
Suppose $0 < T < \infty$. If $u$ is a weak solution to \eqref{eqn:PDE} on $\Omega_{0,T}$, then $(f,g) = \Tr_t u$ is a solution to \eqref{eqn:ODE} on $(0,T)$. Conversely, if $(f,g)$ solves \eqref{eqn:ODE} on $(0,T)$, then there exists a weak solution $u$ to \eqref{eqn:PDE} on $\Omega_{0,T}$ with $\Tr_t u = (f,g)$ for all $t \in (0,T)$.
\end{theorem}

\begin{rem}
It follows immediately that a weak solution to \eqref{eqn:PDE} on $\bbR^n \setminus \{0\}$ is equivalent to a solution to \eqref{eqn:ODE} on $(0,\infty)$. Note that both definitions are local, and involve no boundedness or decay assumptions about the behavior of solutions near $t=0$ or $t=\infty$.
\end{rem}

In general we are interested in solutions to \eqref{eqn:PDE} on the ball $\Omega_T$, not $\Omega_{0,T}$. This requires a further assumption on the asymptotic behavior of $f(t)$ and $g(t)$ as $t\to0$, in order to rule out solutions that are singular at a point. An example of such a solution is $u(x) = |x|^{2-n}$, which is harmonic on $\bbR^n \setminus \{0\}$ but is not contained in $H^1_{\rm loc}(\bbR^n)$ on account of its singular behavior at the origin. The following result can therefore be viewed as a kind of removable singularity theorem.

\begin{theorem}\label{thm:global}
If $(f,g)$ solves \eqref{eqn:ODE} on $(0,T)$, and there exists $p \in \big(0, \frac{n}{2} \big)$ such that
\begin{align}\label{fgbound}
	t^{p} \|f(t)\|_{\Hp} + t^{n-p-1} \|g(t)\|_{\Hm}
\end{align}
is bounded near $t=0$, then there exists a weak solution $u$ to \eqref{eqn:PDE} on $\Omega_T$ with $\Tr_t u = (f,g)$ for all $t \in (0,T)$. Conversely, if $u$ is a weak solution to \eqref{eqn:PDE} on $\Omega_T$, then $(f,g) = \Tr_t u$ is a solution of \eqref{eqn:ODE} on $(0,T)$, with
\begin{align}\label{fgdecay}
	t^{n/2-1} \|f(t)\|_{\Hp} + t^{n/2-1} \|g(t)\|_{\Hm} \to 0
\end{align}
as $t \to 0$, provided $n \geq 3$. When $n=2$ we have
\begin{align}\label{fgdecay2D}
	t^{p} \|f(t)\|_{\Hp} + t^{1-p} \|g(t)\|_{\Hm} \to 0
\end{align}
for any $p \in (0,1)$.
\end{theorem}

In other words, a weak solution on the punctured ball $\Omega_{0,T}$ can be extended to a weak solution on the entire ball $\Omega_T$ if $(f,g) = \Tr_t u$ satisfies the bound \eqref{fgbound}, in which case it necessarily satisfies the decay condition \eqref{fgdecay}. In the special case that $u$ satisfies a linear differential equation, we obtain the stronger result that $\|f(t)\|_{\Hp}$ and $\|g(t)\|_{\Hm}$ are bounded near $t=0$; see Lemma \ref{lemma:lineardecay}.

In this sense the semilinear elliptic equation \eqref{eqn:PDE} is equivalent to the dynamical system \eqref{eqn:ODE}. This correspondence allows us to apply dynamical systems methods to the study of \eqref{eqn:ODE}. A guide as to what can be achieved with this approach comes from the literature of the area known as spatial dynamics, as discussed in the introduction. There are challenges, however, in applying the techniques of spatial dynamics in our setup. 

Spatial dynamics was initiated by the paper of Kirchg\"assner \cite{K82}. The goal of his paper is to establish the existence of a small amplitude solution of a semilinear elliptic equation on a cylindrical domain, which addresses problems that arise in fluid flow. The strategy is to restrict the dynamical system \eqref{eqn:cylinder} to a center manifold. Even though \eqref{eqn:cylinder} is ill-posed, a center manifold theorem can nevertheless be proved, and a reduction to the center manifold leads to a finite-dimensional system, to which bifurcation theory can be applied. This approach can establish the presence of solutions that bifurcate from the trivial (zero) solution. 

The underlying picture to keep in mind is the dynamics near a fixed point in the infinite-dimensional phase space $\cH$. Although the dynamics is not well-posed in either forward or backward time, the splitting of the spectrum, which is unbounded in both directions, into the right and left half planes can be used to get well-posedness in one time direction on appropriate complementing subspaces. Results have been established in this situation which show that there is a splitting into stable/unstable/center manifolds, see \cite{elbialy12,Gallay93}.

Many generalizations of Kirchg\"assner's work have since appeared, notably the work of Mielke \cite{M86}, who was able to characterize {\em all} small bounded solutions in a center-type manifold. An important advance was made by Peterhof, Sandstede and Scheel \cite{PSS97}, who were the first to consider the behavior near a non-trivial solution. They start with a traveling wave solution and consider nearby solutions specifically in the case of time-dependent forcing. They introduce a new approach in their use of the Lyapunov--Schmidt method as an alternative to the center manifold reduction. A key part of their approach is to establish exponential dichotomies as $x\rightarrow \pm \infty$. These are then used to construct stable and unstable manifolds of the fixed point that represents the traveling wave in the infinite-dimensional phase space. A Melnikov method is finally used to establish when these manifolds intersect. 

At the heart of all these pieces of work is the notion that the underlying dynamical system generates a bi-semigroup; see \cite{elbialy12}. The characterization of the dynamics in terms of invariant manifolds can be cast somewhat generally---see \cite{elbialy12,Gallay93} as well as \cite{PSS97}.

\section{Geometric preliminaries}\label{sec:geo}

\subsection{The vector field}\label{sec:flow}
The family of diffeomorphisms in Lemma \ref{lemma:phit} arises as the flow of a suitably chosen vector field. It is more convenient to construct the flow with respect to the variable $\tau = \log t$. This flow, which we denote $\wtp_\tau$, is generated by an autonomous vector field $\wtX$. To motivate our construction we assume that the flow exists and thus obtain some restrictions on the form of $\wtX$, which we then use to construct it explicitly.

If such a flow exists, the fact that $\wtp_\tau$ maps $\pO = \{x : \psi(x) = 1\}$ to $\pO_{\exp(\tau)} = \{x : \psi(x) = e^{2\tau} \}$ would imply $\psi(\wtp_\tau(x)) = e^{2\tau} \psi(x)$, hence
\begin{align}\label{eqn:X}
	\nabla \psi \cdot \wtX = 2 \psi.
\end{align}
Since $\nabla \psi / |\nabla \psi|$ defines a unit normal along each $\pOt$, we conclude that the normal component of $\wtX$ must have magnitude
\[
	\wtX \cdot \frac{\nabla \psi}{|\nabla \psi|} = \frac{2\psi}{|\nabla \psi|},
\]
and so $\wtX$ must be of the form
\[
	\wtX = 2 \psi \frac{\nabla \psi}{|\nabla \psi|^2} + \textrm{tangential part.}
\]
The system of equations \eqref{eqn:ODE} is simplified by choosing a purely normal flow. However, the normal component of $\wtX$ is in general not differentiable at the origin. Therefore, we must include a tangential component in the vector field $\wtX$ in order to obtain a sufficiently smooth flow.

\begin{lemma}\label{lemma:X}
There exists a $C^{1}$ vector field $\wtX$ satisfying \eqref{eqn:X}, with $\wtX(0) = 0$ and $\nabla \wtX(0) = I$. 
\end{lemma}

The vector field $\wtX$ is not uniquely determined---one can add any tangential vector field that is supported away from the origin without changing the above properties. In particular, one can assume that
\[
	\wtX = 2 \psi \frac{\nabla \psi}{|\nabla \psi|^2}
\]
outside an arbitrarily small neighborhood of the origin.

\begin{proof}
Since \eqref{eqn:X} determines the normal component of $\wtX$, we just need to specify the tangential part. For this we take the tangential projection of the vector field $x \mapsto x$. Since the vector field $\nabla \psi / |\nabla \psi|$ is normal to each of the $\pOt$, this projection is given by
\[
	T(x) = x - \frac{\left<x,\nabla \psi(x) \right>}{|\nabla \psi(x)|^2}  \nabla\psi(x).
\]
We then define
\[
	\wtX = 2 \psi \frac{\nabla \psi}{|\nabla \psi|^2} + \chi T,
\]
where $\chi$ is a smooth cut-off function that equals 1 in a small neighborhood of the origin.

Near the origin, where $\chi = 1$, we have
\[
	\wtX(x) = x + \frac{2 \psi(x) - \left<x, \nabla\psi(x) \right>}{|\nabla \psi(x)|^2} \nabla \psi(x).
\]
Since $\psi$ is $C^3$, we can write
\[
	\psi(x) = \frac12 \left<Ax,x\right> + \cO(|x|^3)
\]
and
\[
	\nabla\psi(x) = Ax + \cO(|x|^2)
\]
where the Hessian $A = \nabla^2 \psi(0)$ is positive definite. It follows that
\begin{align}\label{wtX}
	\wtX(x) - x = \frac{2 \psi(x) - \left<x, \nabla\psi(x) \right>}{|\nabla \psi(x)|^2} \nabla \psi(x) = \cO(|x|^2)
\end{align}
for $x$ close to $0$, and so $\wtX$ is differentiable, with $\nabla \wtX(0) = I$.
\end{proof}

Now let $\widetilde{\varphi}_\tau$ denote the flow generated by the vector field $\wtX$ that was constructed in Lemma \ref{lemma:X}. It follows that $\widetilde{\varphi}_\tau$  is defined locally (i.e. for small $\tau$) at each point $x \in \bbR^n$ and is differentiable in $x$. We now prove that this is defined globally.

\begin{lemma}\label{lem:longtime}
The flow $\widetilde\varphi_\tau$ is defined for all $\tau \in \bbR$, and satisfies $\widetilde\varphi_\tau(\Omega) = \Omega_{\exp(\tau)}$.
\end{lemma}

\begin{proof}
Fix $x \in \bbR^n$ and let $\mathcal J \subset \bbR$ denote the maximal interval of existence for $\widetilde{\varphi}_\tau(x)$. Using \eqref{eqn:X} we compute
\[
	\frac{d}{d\tau} \psi(\widetilde{\varphi}_\tau(x)) = \nabla \psi \cdot \widetilde X = 2 \psi(\widetilde\varphi_\tau(x))
\]
for $\tau \in \mathcal J$. It follows that $\psi(\widetilde\varphi_\tau(x)) = c e^{2\tau}$, with $c = \psi(\widetilde\varphi_0(x)) = \psi(x)$, and so $\psi(\widetilde\varphi_\tau(x)) = \psi(x) e^{2\tau}$. Since $\psi$ is proper, this implies that $\widetilde{\varphi}_\tau(x)$ remains bounded for finite $\tau$, and hence is defined for all $\tau \in \bbR$. Recalling the definition of $\Omega_t$ from \eqref{omegadef}, the equality $\psi(\widetilde\varphi_\tau(x)) = \psi(x) e^{2\tau}$ implies
\[
	\widetilde\varphi_\tau(x) \in \Omega_{\exp(\tau)} \Longleftrightarrow \psi(\widetilde\varphi_\tau(x)) < e^{2\tau} \Longleftrightarrow \psi(x) < 1 \Longleftrightarrow x \in \Omega.
\]
This completes the proof.
\end{proof}

To finish the proof of Lemma \ref{lemma:phit} we simply translate the above results from the variable $\tau$ to $t$.

\begin{proof}[Proof of Lemma \ref{lemma:phit}]
For each $t>0$ define $\varphi_t = \widetilde\varphi_{\log t}$. From Lemma \ref{lem:longtime} we obtain
\[
	\varphi_t(\Omega) = \widetilde\varphi_{\log t}(\Omega) = \Omega_t.
\]
Moreover, for any $t_1,t_2>0$ we have
\[
	\varphi_{t_1} \circ \varphi_{t_2} = \widetilde\varphi_{\log t_1} \circ \widetilde\varphi_{\log t_2} =
	\widetilde\varphi_{\log t_1 + \log t_2} = 
	\widetilde\varphi_{\log (t_1 t_2)} = 
	\varphi_{t_1 t_2}
\]
as claimed.
\end{proof}

We conclude this section by giving an explicit formula for the function $\sigma$ defined in \eqref{sigmadef}. For any $x \in \bbR^n$ we have
\[
	X(\varphi_t(x),t) = \frac{d}{dt}\varphi_t(x) = \frac{d}{dt} \ \widetilde\varphi_{\log t}(x) = t^{-1} \wtX(\varphi_t(x)),
\]
and so $X(x,t) = t^{-1} \wtX(x)$. Using the fact that $\nu_x = \nabla \psi(x) / |\nabla\psi(x)|$ and $t(x) = \sqrt{\psi(x)}$, we obtain
\begin{align}\label{sigmaexplicit}
	\sigma(x) = \frac{1}{t(x)} \wtX(x) \cdot \frac{\nabla \psi(x)}{|\nabla\psi(x)|} = 2 \frac{\sqrt{\psi(x)}}{|\nabla \psi(x)|}.
\end{align}

\subsection{Aymptotics}
We now study the asymptotic behavior of $\varphi_t$ and $D\varphi_t$ as $t \to 0$. This will be used in Section \ref{sec:Sobolev}, where we describe the $t$-dependence of the Sobolev spaces $H^s(\Omega_t)$ and $H^s(\pOt)$.

\begin{lemma}\label{lem:lines}
For each $x \in \overline\Omega$ there exists $\widehat x \in \bbR^n$ such that
\[
	\big| \varphi_t(x) - t \widehat x \big| \leq C t^2
\]
as $t \to 0$, for some constant $C$ that does not depend on $x$. Moreover, if $x \neq 0$, then $\widehat x \neq 0$.
\end{lemma}

\begin{proof}
We start by deriving a uniform bound on $\wtp_\tau(x)$. Since $\psi$ has a nondegenerate minimum at $x=0$, there is a constant $c>0$ so that $\psi(x) \geq c|x|^2$ for all $x \in \overline\Omega$. Using the fact that $\psi(x) \leq 1$ for $x \in \overline\Omega$, we thus obtain
\begin{align}\label{phibound}
	c\big| \wtp_\tau(x) \big|^2 \leq \psi(\wtp_\tau(x)) = e^{2\tau} \psi(x) \leq e^{2\tau}
\end{align}
for any $\tau \leq 0$.

Next, recalling the definition of the flow $\widetilde \varphi_\tau$, we compute
\[
	\frac{d}{d\tau} e^{-\tau} \widetilde \varphi_\tau(x) = e^{-\tau} \left( \wtX(\widetilde\varphi_\tau(x)) - \widetilde\varphi_\tau(x) \right) =: E(x,\tau).
\]
It follows from Lemma \ref{lemma:X} and \eqref{phibound} that
\[
	|E(x,\tau)| \leq C e^{-\tau} \big| \wtp_\tau(x) \big|^2 \leq C' e^\tau
\]
hence $E(x,\cdot)$ is integrable on $(-\infty,0]$. Therefore, using the fact that $\widetilde \varphi_0(x) = x$, we have
\begin{align*}
	x - e^{-\tau} \widetilde \varphi_\tau(x) = \int_\tau^0 E(x,s)\, ds = \int_{-\infty}^0 E(x,s)\, ds - \int_{-\infty}^\tau E(x,s)\, ds,
\end{align*}
and so
\[
	\widetilde \varphi_\tau(x) = e^\tau \left(x - \int_{-\infty}^0 E(x,s)\, ds + \int_{-\infty}^\tau E(x,s)\, ds \right).
\]
The desired asymptotic result follows from setting
\[
	\widehat x = x - \int_{-\infty}^0 E(x,s)\, ds
\]
and then observing that the remaining term satisfies
\[
	\left| \int_{-\infty}^\tau E(x,s)\, ds \right| \leq C e^\tau
\]
because $|E(x,s)| \leq C e^s$ uniformly in $x$.

To complete the proof, suppose that $\widehat x = 0$, and hence $\big| \widetilde \varphi_\tau(x) \big| \leq C e^{2\tau}$.
Since $\psi$ has a critical point at $x=0$, it satisfies $\psi(x) \leq C' |x|^2$ for some positive constant $C'$. As in \eqref{phibound}, we obtain
\[
	\psi(x) e^{2\tau} = \psi(\wtp_\tau(x)) \leq C' \big| \widetilde \varphi_\tau(x) \big|^2.
\]
This implies
\[
	\psi(x) e^{2\tau} \leq C' \big| \widetilde \varphi_\tau(x) \big|^2 \leq C^2 C' e^{4\tau}.
\]
and hence $\psi(x) \leq C^2 C' e^{2\tau}$. Letting $\tau \to -\infty$, we obtain $\psi(x) = 0$, and so $x=0$.
\end{proof}

In other words, the trajectories of the flow are asymptotic to straight lines for small $t$. We now use this to prove that the functions $\{\sigma_t\}$ defined in \eqref{sigmatdef} converge uniformly as $t \to 0$.

\begin{lemma}\label{lem:sigmauniform}
There is a positive function $\sigma_0 \colon \pO \to \bbR$ such that $\sigma_t \to \sigma_0$ uniformly as $t \to 0$.
\end{lemma}

\begin{proof}
Let $A = \nabla^2 \psi(0)$. For small $x$ we have
\[
	\psi(x) = \frac12\left<Ax,x\right> + \cO(|x|^3)
\]
and
\[
	|\nabla \psi(x)| = |Ax| + \cO(|x|^2).
\]
Combining this with \eqref{sigmaexplicit}, we see that
\[
	\sigma(x) = 2 \frac{\sqrt{\psi(x)}}{|\nabla \psi(x)|} = \frac{\sqrt{2\left<Ax,x\right>}}{|Ax|} + \cO(|x|).
\]
Now let $y \in \pO$. From Lemma \ref{lem:lines} we have $\varphi_t(y) = t \widehat y + \cO(t^2)$ for some nonzero $\widehat y \in \bbR^n$, and so
\[
	\sigma_t(y) = \sigma(\varphi_t(y)) = \frac{\sqrt{2\left<A\widehat y,\widehat y\right>}}{|A\widehat y|} + \cO(t).
\]
We thus define $\sigma_0(y) = \sqrt{2\left<A\widehat y,\widehat y\right>} / |A\widehat y|$. Since the constant $C$ in Lemma \ref{lem:lines} is independent of $x$, we conclude that $\sigma_t \to \sigma_0$ uniformly on $\pO$.
\end{proof}

We next consider the Jacobian matrix $D\varphi_t(x)$ and its determinant.

\begin{lemma}\label{lem:determinant}
There exist positive constants $c_1, c_2$ such that
\[
	c_1 t^n \leq \det(D\varphi_t(x)) \leq c_2 t^n
\]
for all $x \in \overline\Omega$ and sufficiently small $t>0$.
\end{lemma}

\begin{proof}
Differentiating the flow equation
\[
	\frac{d}{d\tau} \widetilde\varphi_\tau(x) = \widetilde X(\widetilde\varphi_\tau(x))
\]
with respect to $x$, we find that $D \widetilde\varphi_\tau(x)$ satisfies the linear system
\begin{align}\label{eqn:Dphi}
	\frac{d}{d\tau} D \widetilde\varphi_\tau(x) = \big[\nabla \widetilde X (\widetilde\varphi_\tau(x))\big] D \widetilde\varphi_\tau(x).
\end{align}
Using Jacobi's formula we obtain
\begin{align*}
	\frac{d}{d\tau} \log \det(D \widetilde\varphi_\tau(x)) &= \tr\left( D \widetilde\varphi_\tau(x)^{-1} \big[\nabla \widetilde X (\widetilde\varphi_\tau(x))\big] D \widetilde\varphi_\tau(x)\right) \\
	&= (\nabla \cdot \wtX)(\widetilde\varphi_\tau(x)).
\end{align*}
From Lemma \ref{lemma:X} and \eqref{phibound}, the divergence satisfies 
\[
	(\nabla \cdot \wtX)(\widetilde\varphi_\tau(x)) = n + \cO(e^\tau).
\]
Since $\log \det(D \widetilde\varphi_0(x)) = 0$, we find that
\[
	n\tau - C \leq \log \det(D \widetilde\varphi_\tau(x)) \leq n\tau + C
\]
for all $\tau \leq 0$, where $C$ does not depend on $x$. It follows that
\[
	e^{-C} e^{n\tau} \leq \det(D \widetilde\varphi_\tau(x)) \leq e^C e^{n\tau}
\]
uniformly in $x$.
\end{proof}

\begin{lemma}\label{lem:Jacobian}
For each $x \in \overline\Omega$ there exists an invertible matrix $M(x)$ such that
\[
	\big\| D \varphi_t(x) - t M(x) \big\| \leq C t^2
\]
as $t \to 0$, for some constant $C$ that does not depend on $x$. Moreover, $\|M(x)\|$ and $\|M(x)^{-1}\|$ are  bounded above uniformly in $x$.
\end{lemma}

\begin{proof}
Using \eqref{eqn:Dphi} we find that
\begin{align}
	\frac{d}{d\tau} e^{-\tau} D \widetilde\varphi_\tau(x) = e^{-\tau} \left( \nabla \widetilde X(\widetilde\varphi_\tau(x)) - I \right) D \widetilde\varphi_\tau(x).
\end{align}
Integrating from $\tau$ to 0 and using the fact that $D\widetilde\varphi_0(x) = I$, we obtain
\[
	\big\| e^{-\tau} D \widetilde\varphi_\tau(x) \big\| \leq 1 + \int_\tau^0 \left\| \nabla \widetilde X(\widetilde\varphi_s(x)) - I \right\|
	\big\| e^{-s} D \widetilde\varphi_s(x) \big\| \,ds.
\]
From Lemma \ref{lemma:X} and \eqref{phibound} we have
\[
	\left\| \nabla \widetilde X(\widetilde\varphi_\tau(x)) - I \right\| \leq C e^\tau,
\]
where $C$ does not depend on $x$. It follows from Gronwall's inequality that
\begin{align}\label{eq:Gronwall}
	\big\| e^{-\tau} D \widetilde\varphi_\tau(x) \big\| \leq \exp \left\{ \int_\tau^0 C e^s ds \right\} \leq e^C
\end{align}
for any $\tau \leq 0$.

Now define
\[
	E(x,\tau) = e^{-\tau} \left( \nabla \widetilde X(\widetilde\varphi_\tau(x)) - I \right) D \widetilde\varphi_\tau(x).
\]
It follows from \eqref{eq:Gronwall} that $E(x,\cdot)$ is integrable on $(-\infty,0]$, so we can integrate \eqref{eqn:Dphi} to obtain
\[
	D \widetilde\varphi_\tau(x) = e^\tau \left(I - \int_{-\infty}^0 E(x,s)\, ds + \int_{-\infty}^\tau E(x,s)\, ds\right).
\]
We thus define
\[
	M(x) = I - \int_{-\infty}^0 E(x,s)\, ds.
\]
Bounding the remaining term as in the proof of Lemma \ref{lem:lines}, it follows that $\big\| D \varphi_t(x) - t M(x) \big\| \leq C t^2$. In particular, this implies $t^{-1} D\varphi_t(x) \to M(x)$ as $t \to 0$. From the estimate in Lemma \ref{lem:determinant} we see that $\det(t^{-1} D\varphi_t(x))$ is bounded away from zero, and so the limit $M(x)$ is invertible.
\end{proof}

\subsection{Mean curvature and the first variation of area}
The rate of change of the area of $\pOt$ is related to its mean curvature. The mean curvature of a hypersurface is defined to be the divergence of the outward unit normal, and so for $\pOt$ we have
\begin{align}\label{Hdef}
	H_{\pOt} = \nabla \cdot \left( \frac{\nabla \psi}{|\nabla \psi|} \right).
\end{align}
In the radial case, where $\psi(x) = |x|^2$, one simply has $H_{\pOt} = (n-1)/t$ for all $x \in \pOt$. An overview of mean curvature and level set methods can be found in \cite{CM16}.

To study the $t=0$ limit of \eqref{eqn:ODE}, we must understand the asymptotic behavior of the function $H_t = H_{\pOt} \circ \varphi_t \big|_\pO$. Using the nondegeneracy assumption imposed on $\psi$ in Hypothesis \ref{hyp:geo}, we can control the mean curvature for small $t$.

\begin{lemma}\label{lem:Hlimit}
There is a function $H_0 \colon \pO \to \bbR$ such that $t H_t \to H_0$ uniformly as $t \to 0$.
\end{lemma}

\begin{proof}
Calculating the divergence of $\nabla\psi / |\nabla\psi|$, we find
\[
	H_{\pOt} = \frac{1}{|\nabla \psi|} \left(\Delta \psi  - \nabla^2 \psi \left( \frac{\nabla\psi}{|\nabla\psi|}, \frac{\nabla\psi}{|\nabla\psi|} \right) \right).
\]
Near the origin we have 
\[
	\nabla \psi(x) = Ax + \cO(x^2)
\]
and
\[
	\nabla^2 \psi(x) = A + \cO(x).
\]
It follows that
\[
	H_{\pOt}(x) = \frac{1}{|Ax|} \left( \tr A - \frac{\left<A^2x,Ax\right>}{|Ax|^2} \right) + \cO(1).
\]
Now let $y \in \pO$. From Lemma \ref{lem:lines} we have $\varphi_t(y) = t \widehat y + \cO(t^2)$ for some nonzero $\widehat y \in \bbR^n$, and so
\[
	H_t(y) = H_{\pOt}(\varphi_t(y)) = \frac{1}{t} \frac{1}{|A \widehat y|} \left( \tr A - \frac{\left<A^2\widehat y,A\widehat y\right>}{|A\widehat y|^2} \right) + \cO(1).
\]
We thus define 
\[
	H_0(y) = \frac{1}{|A \widehat y|} \left( \tr A - \frac{\left<A^2\widehat y,A\widehat y\right>}{|A\widehat y|^2} \right),
\]
so that $t H_t(y) = H_0(y) + \cO(t)$, where the error term is uniform in $y$. This completes the proof.
\end{proof}

Next, let $d\mu_t$ and $d\mu$ denote the surface measures on $\pOt$ and $\pO$, respectively, and let $a_t \colon \pO \to \bbR$ denote the Radon--Nikodym derivative of the pulled-back measure $\varphi_t^* d\mu_t$ with respect to $d\mu$, so that $\varphi_t^* d\mu_t = a_t \,d\mu$. By definition, this means
\begin{align}\label{atdef}
	\int_{\pOt} w \, d\mu_t = \int_\pO (w \circ \varphi_t)   a_t  \,d\mu
\end{align}
for any measurable function $w$ on $\pOt$. This can be computed explicitly as the Jacobian determinant $|\det(D \varphi_t^\p)|$, where $\varphi_t^\p \colon \pO \to \pOt$ denotes the restriction of $\varphi_t$ to the boundary of the reference domain.

\begin{lemma}\label{lem:firstvar}
The function $a_t$ satisfies
\[
	\frac{d a_t}{dt} = a_t \left[ \sigma_t H_t + \big(\dv^{\pOt} \gamma\big)\circ \varphi_t \right]
\]
for all $t>0$.
\end{lemma}

The proof can be found in \cite[Section 1.3]{CM11}. 
Using this, we can describe the asymptotic behavior of the area function $a_t$. This is a more delicate quantity than the total area of $\pOt$, and is quite sensitive to the behavior of the vector field $X$ near the origin.

\begin{lemma}\label{lem:atbound}
There exist positive constants $c_1$ and $c_2$ so that
\begin{align}\label{atbound}
	c_1 t^{n-1} \leq a_t(y) \leq c_2 t^{n-1}
\end{align}
for all $y \in \pO$ and $t>0$ sufficiently small.
\end{lemma}

\begin{proof}
Writing $X = \sigma \nu + \gamma$ and using \eqref{div:nortan}, we obtain
\[
	(\dv X )\big|_{\pOt} = \dv^{\pOt} \gamma + \sigma H_{\pOt} + \nu \cdot \nabla_\nu X,
\]
and hence
\begin{align*}
	\sigma_t H_t + \big(\dv^{\pOt} \gamma\big)\circ \varphi_t &= \left( \sigma H_{\pOt} + \dv^{\pOt} \gamma\right) \circ \varphi_t \\
	&= \left(\dv X - \nu \cdot \nabla_\nu X \right) \circ \varphi_t.
\end{align*}
From Lemma \ref{lemma:X} we have $\nabla \wtX = I + \cO(|x|)$. This implies $\dv \wtX = n + \cO(|x|)$ and $\nu \cdot \nabla_\nu \wtX = 1 + \cO(|x|)$, hence $\dv \wtX - \nu \cdot \nabla_\nu \wtX = (n-1) + \cO(|x|)$. Since $X(x,t) = t^{-1} \wtX(x)$, we obtain
\[
	\sigma_t H_t + \big(\dv^{\pOt} \gamma\big)\circ \varphi_t = \frac{n-1}{t} + \frac{\cO(|\varphi_t(y)|)}{t}
	= \frac{n-1}{t} + \cO(1),
\]
using Lemma \ref{lem:lines} to bound $\varphi_t(y)$ for $y \in \pO$. Therefore
\[
	\frac{n-1}{t} - C \leq \frac{1}{a_t} \frac{d a_t}{dt} \leq \frac{n-1}{t} + C
\]
uniformly on $\pO$, and the result follows.
\end{proof}

\subsection{The coarea formula}
When relating a function $u$ and its boundary data $f(t)$ and $g(t)$, we will make frequent use of the coarea formula. This allows us to relate the integral of $u$ over a given domain to the integrals of $u$ over the level sets of a sufficiently smooth function. It can be viewed as a generalization to the nonradial case of the standard formula for integration in polar coordinates.

Suppose $\Psi \colon \bbR^n \to \bbR$ is smooth. Sard's theorem implies that for almost every $t\in\bbR$, the level set $\Psi^{-1}(t)$ is a smooth hypersurface. Let $d\mu_t$ denote the induced measure on $\Psi^{-1}(t)$. Defining the region $\Omega_{a,b} = \{a < \Psi(x) < b\}$, the coarea formula says that
\[
	\int_{\Omega_{a,b}} w |\nabla \Psi| = \int_a^b \left( \int_{\Psi^{-1}(t)} w\, d\mu_t \right) dt
\]
for any measurable function $w$ that is either nonnegative or integrable \cite{C06}. In fact, if $d\mu_t$ is suitably interpreted, one only requires the function $\Psi$ to be Lipschitz; see \cite{EG92} for a general version of this result. If $w / |\nabla \Psi|$ is nonnegative or integrable, we have
\[
	\int_{\Omega_{a,b}} w = \int_a^b \left( \int_{\Psi^{-1}(t)} \frac{w}{|\nabla\Psi|} \, d\mu_t \right) dt.
\]

To relate this to the domain $\Omega_{a,b} = \{a^2 < \psi(x) < b^2\}$ defined in \eqref{annulus}, we let $\Psi = \sqrt{\psi}$ and calculate $\nabla \Psi = \nabla \psi / (2 \sqrt{\psi})$. Comparing with \eqref{sigmaexplicit}, we have $|\nabla \Psi| = \sigma^{-1}$, and so the coarea formula yields
\[
	\int_{\Omega_{a,b}} w = \int_a^b \left( \int_{\pOt} \sigma w\, d\mu_t \right) dt.
\]
Finally, using the fact that $\sigma_t = \sigma \circ \varphi_t\big|_\pO$ and recalling the definition of $a_t$ from \eqref{atdef}, we obtain
\begin{align}\label{coarea}
	\int_{\Omega_{a,b}} w = \int_a^b \left( \int_{\pO} \sigma_t  (w \circ \varphi_t) a_t \,d\mu \right) dt.
\end{align}
Note that all of the integrals on the right-hand side are computed on the fixed hypersurface $\pO$.

\subsection{Scaling of Sobolev norms}\label{sec:Sobolev}
The diffeomorphisms $\varphi_t \colon \Omega \to \Omega_t$ induce maps $H^s(\Omega_t) \to H^s(\Omega)$ and $H^s(\pOt) \to H^s(\pO)$ via the pullback, $u \mapsto u \circ \varphi_t$. To prove Theorem \ref{thm:global} we will need estimates on the norms of these maps for small $t$.

\begin{lemma}\label{lem:Sobolev}
There exist constants $c_1$ and $c_2$ such that the following estimates hold for small $t$:
\[
	c_1 t^{n/2} \|u \circ \varphi_t\|_{L^2(\Omega)}  \leq \|u \|_{L^2(\Omega_t)} \leq c_2 t^{n/2} \|u \circ \varphi_t \|_{L^2(\Omega)}
\]
for all $u \in L^2(\Omega_t)$,
\[
	c_1  t^{n/2}  \|u \circ \varphi_t \|_{H^1(\Omega)} \leq \|u\|_{H^1(\Omega_t)} \leq c_2  t^{n/2 - 1}  \|u \circ \varphi_t \|_{H^1(\Omega)}
\]
for all $u \in H^1(\Omega_t)$,
\[
	c_1 t^{(n-1)/2} \| f \|_{L^2(\pO)}  \leq \left\| f \circ \varphi_t^{-1} \right\|_{L^2(\pOt)} \leq c_2 t^{(n-1)/2} \| f \|_{L^2(\pO)}
\]
for all $f \in L^2(\pO)$, and
\[
	c_1 t^{(n-1)/2} \left\| f \right\|_{H^1(\pO)} \leq \left\| f \circ \varphi_t^{-1} \right\|_{H^1(\pOt)} \leq c_2 t^{(n-3)/2} \left\| f \right\|_{H^1(\pO)}
\]
for all $f \in H^1(\pO)$.
\end{lemma}

More precisely, for any $T>0$ there exist constants $c_1(T)$ and $c_2(T)$ such that the above estimates hold for all $t \in (0,T]$.

\begin{proof}
For the $L^2(\Omega_t)$ estimate we compute
\begin{align*}
	\int_{\Omega_t} u^2 
	 &= \int_\Omega (u \circ \varphi_t)^2 \det (D \varphi_t) 
\end{align*}
and then apply Lemma \ref{lem:determinant}. The $L^2(\pO)$ estimate is obtained similarly, writing
\[
	\int_{\pOt} \big(f \circ \varphi_t^{-1}\big)^2 d\mu_t = \int_{\pO} f^2 a_t \,d\mu
\]
and then using \eqref{atbound}.

For the $H^1(\Omega_t)$ estimate we first compute $\nabla(u \circ \varphi_t) = (D \varphi_t)^T (\nabla u) \circ \varphi_t$. It follows from Lemma~\ref{lem:Jacobian} that
\begin{align}\label{eq:gradbound}
	c_1 t \big|(\nabla u) \circ \varphi_t \big| \leq |\nabla(u \circ \varphi_t)| \leq c_2 t \big|(\nabla u) \circ \varphi_t \big|
\end{align}
and so the norm of the gradient
\[
	\| \nabla u \|_{L^2(\Omega_t)}^2 = \int_{\Omega_t} |\nabla u|^2
	= \int_\Omega \left |(\nabla u) \circ \varphi_t \right|^2 \det (D \varphi_t)
\]
satisfies the estimate
\begin{align}
	c_1 t^{n-2} \| \nabla(u \circ \varphi_t)\|^2_{L^2(\Omega)} \leq \| \nabla u \|_{L^2(\Omega_t)}^2 \leq c_2 t^{n-2} \| \nabla(u \circ \varphi_t)\|^2_{L^2(\Omega)}.
\end{align}
Combining this with the $L^2(\Omega_t)$ estimate, we have
\begin{align*}
	\|u\|_{H^1(\Omega_t)}^2 &= \|u\|_{L^2(\Omega_t)}^2 + \|\nabla u\|_{L^2(\Omega_t)}^2 \\
	&\leq c_2 \left(t^n \|u \circ \varphi_t \|^2_{L^2(\Omega)} + t^{n-2} \|\nabla(u \circ \varphi_t) \|^2_{L^2(\Omega)} \right) \\ 
	&\leq c_2 t^{n-2} \left(\|u \circ \varphi_t \|^2_{L^2(\Omega)} + \|\nabla(u \circ \varphi_t) \|^2_{L^2(\Omega)} \right) \\ 
	&= c_2 t^{n-2} \|u \circ \varphi_t \|^2_{H^1(\Omega)}
\end{align*}
and
\begin{align*}
	\|u\|_{H^1(\Omega_t)}^2  &\geq c_1 \left(t^n \|u \circ \varphi_t \|^2_{L^2(\Omega)} + t^{n-2} \|\nabla(u \circ \varphi_t) \|^2_{L^2(\Omega)} \right) \\ 
	&\geq c_1 t^{n} \left(\|u \circ \varphi_t \|^2_{L^2(\Omega)} + \|\nabla(u \circ \varphi_t) \|^2_{L^2(\Omega)} \right) \\ 
	&= c_1 t^{n} \|u \circ \varphi_t \|^2_{H^1(\Omega)}
\end{align*}
as desired.

Finally, for the $H^1(\pO)$ estimate, we recall that the tangential gradient $\nabla^{\pO} f$ is given by $\nabla^{\pO} f = \nabla \hat f - (\p \hat f/\p\nu) \nu$, where $\hat f$ is any extension of $f$ to a neighborhood of $\pO$. Choosing an extension $\hat f$ with $\p \hat f/\p\nu = 0$, we use \eqref{eq:gradbound} to compute
\begin{align*}
	\left\| \nabla^{\pOt} \big(f \circ \varphi_t^{-1}\big)\right\|_{L^2(\pOt)}^2 &= \int_{\pOt} \big| \nabla^{\pOt} \big(f \circ \varphi_t^{-1}\big)\big|^2 d\mu_t \\
	&\leq \int_{\pOt} \big| \nabla \big(\hat f \circ \varphi_t^{-1}\big)\big|^2 d\mu_t \\
	&\leq \frac{C}{t^2} \int_{\pOt} \big| (\nabla \hat f) \circ \varphi_t^{-1}\big|^2 d\mu_t \\
	&= \frac{C}{t^2} \int_{\pO} \big| \nabla \hat f \big|^2 a_t\, d\mu \\
	&\leq C t^{n-3} \int_{\pO} \big| \nabla \hat f \big|^2 \, d\mu \\
	&= C t^{n-3} \| \nabla^{\pO} f \|_{L^2(\pO)}^2
\end{align*}
where in the last line we have used the fact that $\nabla^{\pO} f = \nabla \hat f$ for this particular choice of $\hat f$. Similarly, choosing an extension $\hat f$ of $f$ so that $\nabla^{\pOt} \big(f \circ \varphi_t^{-1}\big) = \nabla \big( \hat f \circ \varphi_t^{-1}\big)$ on $\pOt$, we find that
\begin{align*}
	\left\| \nabla^{\pOt} \big(f \circ \varphi_t^{-1}\big)\right\|_{L^2(\pOt)}^2 \geq
	C t^{n-3} \| \nabla^{\pO} f \|_{L^2(\pO)}^2
\end{align*}
for some different constant $C$. Combining this with the already obtained estimate for the $L^2(\pO)$ norm, the result follows.
\end{proof}

\section{Evolution of the boundary data}\label{sec:evolution}
In this section we prove Theorems \ref{thm:local} and \ref{thm:global}, which say that the partial differential equation \eqref{eqn:PDE} is equivalent to the system of ordinary differential equations \eqref{eqn:ODE} for the boundary data.  Aside from issues of regularity, the proof of Theorem \ref{thm:local} consists of direct computations using integration by parts and the coarea formula \eqref{coarea}. The proof of Theorem \ref{thm:global}, on the other hand, is more involved, and requires a detailed understanding of the geometry of the level sets $\pOt$ as $t \to 0$.

\subsection{An approximation argument}
To prove the second statement in Theorem \ref{thm:local} we need to reconstruct the function $u \in H^1_{\rm loc}(\Omega_{0,T})$ from its Cauchy data $(f(t),g(t))$ for $0 < t < T$. This is made possible by the following result.

\begin{prop}\label{prop:approximation}
Suppose $\Omega \subset \bbR^n$ is a bounded domain, with Lipschitz boundary $\pO$. Then $C^1([a,b], C^1(\pO))$ is dense in $C^0([a,b], H^1(\pO)) \cap C^1([a,b], L^2(\pO))$.
\end{prop}

That is, if $f \in C^0([a,b], H^1(\pO)) \cap C^1([a,b], L^2(\pO))$, there exist approximating functions $f_\epsilon \in C^1([a,b], C^1(\pO))$ such that
\[
	\|f_\epsilon(t) - f(t)\|_{H^1(\pO)} + \left\| f_\epsilon'(t) - f'(t) \right\|_{L^2(\pO)} \to 0
\]
uniformly in $t$ as $\epsilon \to 0$. The main ingredient in the proof is the following lemma, which combines a standard mollification argument in local coordinates with a version of Kolmogorov's compactness criteria; cf. \cite{K31,W87}.

\begin{lemma}\label{lem:moll}
Suppose $f \in H^k(\pO)$, with $k \in \{0,1\}$. There exist functions $f_\epsilon \in C^1(\pO)$ such that $\|f_\epsilon - f\|_{H^k(\pO)} \to 0$ as $\epsilon \to 0$. Moreover, the convergence of $f_\epsilon$ to $f$ is uniform on precompact sets of $H^k(\pO)$. That is, if $S \subset H^k(\pO)$ has compact closure, then for any $\delta>0$ there exists $\epsilon_0>0$ such that
\[
	\|f_\epsilon - f\|_{H^k(\pO)} < \delta
\]
for all $\epsilon < \epsilon_0$ and all $f \in S$.
\end{lemma}

The density of $C^1(\pO)$ in $L^2(\pO)$ and $H^1(\pO)$ is standard. The key to the proof of the above lemma is to construct the approximating functions $f_\epsilon$ in an explicit way that yields uniform convergence on precompact subsets of $H^k(\pO)$.

\begin{proof}
We first recall the definition of $H^s(\pO)$ for a Lipschitz domain, following \cite{M00}: There exist two finite collections of open sets, $\{W_j\}$ and $\{\Omega_j\}$, such that $\pO \subset \cup_j W_j$, $\Omega \cap W_j = \Omega_j \cap W_j$ for each $j$, and each $\Omega_j$ is given (after a rigid motion) by the hypograph of a Lipschitz function $\zeta_j \colon \bbR^{n-1} \to \bbR$. By this we mean that there is a rigid motion $\kappa_j$ of $\bbR^n$ so that
\[
	\kappa_j(\Omega_j) = \{ x = (x',x_n) \in \bbR^n : x_n < \zeta_j(x') \}.
\]
Let $\{\phi_j\}$ be a partition of unity subordinate to the covering $\{W_j\}$. Given a function $f \colon \pO \to \bbR$, we define functions $f_j \colon \bbR^{n-1} \to \bbR$ by
\begin{align}\label{fjdef}
	f_j(x') = \big(\phi_j f\big) \big(\kappa_j^{-1}(x', \zeta_j(x')) \big).
\end{align}
We then define the $H^k(\pO)$ Sobolev norm by
\begin{align}\label{Hkdef}
	\|f\|_{H^k(\pO)} = \sum_j \|f_j\|_{H^k(\bbR^{n-1})}.
\end{align}

We are now ready to define the mollification of $f \in H^k(\pO)$. We start by inverting \eqref{fjdef} as follows. If $x \in \pO \cap W_j$, then $\kappa_j(x) = (x', \zeta_j(x'))$ for a unique $x' \in \bbR^{n-1}$, namely $x' = P \kappa_j(x)$, where $P \colon \bbR^n \to \bbR^{n-1}$ denotes projection onto the first $n-1$ components. It follows that $(\phi_j f)(x) = f_j\big( P \kappa_j(x) \big)$ for any $x \in \pO \cap W_j$, and so
\[
	f(x) = \sum_j f_j\big( P \kappa_j(x) \big)
\]
for each $x \in \pO$. Now, letting $\eta_\epsilon$ denote the standard mollifier in $\bbR^{n-1}$, we set
\[
	f_\epsilon(x) = \sum_j (\eta_\epsilon \ast f_j)\big( P \kappa_j(x) \big).
\]
It follows from \eqref{Hkdef} and standard properties of $\eta_\epsilon$ that
\begin{align}\label{mollbound}
	\|f_\epsilon\|_{H^k(\pO)} \leq C \|f \|_{H^k(\pO)}
\end{align}
for some constant $C$ that does not depend on $f$ or $\epsilon$, and
\begin{align}\label{mollconverge}
	\|f_\epsilon - f\|_{H^k(\pO)} \to 0
\end{align}
as $\epsilon \to 0$. This completes the first part of the proof.

We prove the second claim by contradiction. Suppose there exists a number $\delta_0 > 0$, a sequence of positive numbers $\epsilon_n$ tending to zero, and functions $f^{(n)} \in S$ such that
\[
	\|f^{(n)}_{\epsilon_n} - f^{(n)}\|_{H^k(\pO)} \geq \delta_0
\]
for all $n$. Using \eqref{mollbound} we obtain
\begin{align*}
	\delta_0 &\leq \|f^{(n)}_{\epsilon_n} - f^{(n)}\|_{H^k(\pO)} \\
	& \leq \|f^{(n)}_{\epsilon_n} - f_{\epsilon_n} \|_{H^k(\pO)} + \|f_{\epsilon_n}  - f\|_{H^k(\pO)} + \|f  - f^{(n)}\|_{H^k(\pO)} \\
	& \leq (1+C) \|f  - f^{(n)}\|_{H^k(\pO)} + \|f_{\epsilon_n}  - f\|_{H^k(\pO)}
\end{align*}
for any function $f \in H^k(\pO)$. Since $\|f_{\epsilon_n}  - f\|_{H^k(\pO)} \to 0$ as $n \to \infty$, we have
\[
	\liminf_{n\to\infty} \|f  - f^{(n)}\|_{H^k(\pO)} \geq \frac{\delta_0}{1+C},
\]
which shows that $f^{(n)}$ has no convergent subsequences, contradicting the hypothesis on $S$.
\end{proof}

We are now ready to prove the proposition.

\begin{proof}[Proof of Proposition \ref{prop:approximation}]
Given $f \in C^0([a,b], H^1(\pO)) \cap C^1([a,b], L^2(\pO))$, we use the construction of Lemma \ref{lem:moll} to define $f_\epsilon$ pointwise in $t$, i.e. $f_\epsilon(t) = f(t)_\epsilon$ for each $t \in [a,b]$. It follows from \eqref{mollbound} that $f_\epsilon \in C^0([a,b], H^1(\pO))$ for each $\epsilon$, and \eqref{mollconverge} implies that $f_\epsilon(t) \to f(t)$ in $H^1(\pO)$ for each $t$. Since $\{f(t) : t \in [a,b]\}$ is a compact subset of $H^1(\pO)$, the convergence is in fact uniform in $t$, hence $f_\epsilon \to f$ in $C^0([a,b], H^1(\pO))$. Moreover, since
\[
	\left\| \frac{f_\epsilon(t+h) - f_\epsilon(t)}{h} - (f'(t))_\epsilon \right\|_{L^2(\pO)} \leq C
	\left\| \frac{f(t+h) - f(t)}{h} - f'(t) \right\|_{L^2(\pO)}
\]
for any $h>0$, we conclude that $f_\epsilon$ is differentiable in $t$, with
\[
	(f_\epsilon)' = (f')_\epsilon \in C^0([a,b],L^2(\pO))
\]
and $f_\epsilon'(t) \to f'(t)$ in $L^2(\pO)$, where the convergence is again uniform in $t$.
\end{proof}

\subsection{Preliminary constructions}
We now use Proposition \ref{prop:approximation} to reconstruct $u$ from its Cauchy data.

First suppose $f \in C^0([a,b],C^0(\pO))$ for some $0 < a < b < \infty$. For each $x \in \Omega_{a,b}$ there is a unique $t \in (a,b)$ and $y \in \pO$ such that $x = \varphi_t(y)$, namely $t = \sqrt{\psi(x)}$ and $y = \varphi_t^{-1}(x)$. Thus we can define a continuous function $u\colon \Omega_{a,b} \to \bbR$ by
\begin{align}\label{udef}
	u(x) = f(t)\big(\varphi_{t(x)}^{-1}(x) \big).
\end{align}

We first relate the integrability properties of $u$ to those of $f$.

\begin{lemma}\label{lem:L2}
There exists a constant $C = C(a,b)$ such that $\|u\|_{L^2(\Omega_{a,b})} \leq C \|f\|_{C^0([a,b],L^2(\pO))}$ for all $f \in C^0([a,b],C^0(\pO))$. Therefore, the map $f \mapsto u$ in \eqref{udef} extends uniquely to a bounded operator $C^0([a,b], L^2(\pO)) \to L^2(\Omega_{a,b})$.
\end{lemma}

\begin{proof}

Let $f \in C^0([a,b],C^0(\pO))$. From the definition of $u$ and the coarea formula \eqref{coarea} we have
\begin{align*}
	\int_{\Omega_{a,b}} u^2 
	&= \int_a^b \left( \int_\pO \sigma_t f(t)^2 a_t \,d\mu \right) dt \\
	&\leq C \sup_{a \leq t \leq b} \| f(t) \|_{L^2(\pO)}^2,
\end{align*}
since $\sigma_t$ and $a_t$ are bounded uniformly for $t \in [a,b]$. The existence of a unique bounded extension follows from the density of $C^0([a,b],C^0(\pO))$ in $C^0([a,b], L^2(\pO))$, using Proposition \ref{prop:approximation}.
\end{proof}

We next examine the differentiability properties of $u$.

\begin{lemma}\label{lem:H1}
If $f \in C^0([a,b],H^1(\pO)) \cap C^1([a,b],L^2(\pO))$, then $u \in H^1(\Omega_{a,b})$ and the weak derivative is given by
\begin{align}\label{eqn:weak}
	\nabla u \big|_{\pOt} = \nabla^{\pOt} \left(f \circ \varphi_t^{-1}\right) + \sigma^{-1} \nu \left( \frac{df}{dt} \circ \varphi_t^{-1} - \gamma \cdot \nabla^{\pOt} \left(f \circ \varphi_t^{-1} \right) \right) \in L^2(\pOt)
\end{align}
for $a < t < b$.
\end{lemma}

\begin{proof}
We again use a density argument based on Proposition \ref{prop:approximation}. If $f \in C^0([a,b],C^1(\pO)) \cap C^1([a,b],C^0(\pO))$, then $u \in C^1(\Omega_{a,b})$. Differentiating the equation $f = u \circ \varphi_t$, we obtain
\begin{align*}
	\frac{df}{dt} 
	&= \big( X \cdot \nabla u \big) \circ \varphi_t \\
	&= \big( \gamma \cdot \nabla u + \sigma \nu \cdot \nabla u \big) \circ \varphi_t \\
	&= \left( \gamma \cdot \nabla^{\pOt} \left(f \circ \varphi_t^{-1}\right) + \sigma \frac{\p u}{\p \nu} \right) \circ \varphi_t
\end{align*}
and so the normal derivative of $u$ can be computed in terms of $f$ as
\[
	\frac{\p u}{\p\nu} = \sigma^{-1} \left(\frac{df}{dt} \circ \varphi_t^{-1} - \gamma \cdot \nabla^{\pOt} \left(f \circ \varphi_t^{-1}\right)\right).
\]
Decomposing $\nabla u$ into normal and tangential components along $\pOt$,
\begin{align*}
	\nabla u \big|_{\pOt} = \nabla^{\pOt} u + \frac{\p u}{\p \nu} \nu,
\end{align*}
we arrive at \eqref{eqn:weak}.

Next, using \eqref{eqn:weak}, Lemma \ref{lem:Sobolev}, and the fact that $\sigma$ and $\gamma$ are uniformly bounded on $\Omega_{a,b}$, we find that
\[
	\left\| \nabla u \big|_{\pOt} \right\|_{L^2(\pOt)} \leq C \left( \left\| \frac{df}{dt}\right\|_{L^2(\pO)} + \|f(t)\|_{H^1(\pO)} \right)
\]
for some constant $C = C(a,b)$. It then follows from the coarea formula, as in the proof of Lemma~\ref{lem:L2}, that
\[
	\| \nabla u \|_{L^2(\Omega_{a,b})} \leq C \sup_{a \leq t \leq b} \left( \left\| \frac{df}{dt}\right\|_{L^2(\pO)} + \|f(t)\|_{H^1(\pO)} \right).
\]
The result now follows from Proposition \ref{prop:approximation}.
\end{proof}

We are now ready to prove the main result of this section, which will allow us to describe a weak solution to \eqref{eqn:PDE} in terms of its restriction to each hypersurface $\pOt$.

\begin{lemma}\label{lem:weaksol}
If $f \in C^0([a,b],\Ht) \cap C^1([a,b],\Hp) \cap C^2([a,b],\Hm)$, 
and $g$ is defined by
\[
	g = \sigma_t^{-1} \left( \frac{df}{dt} - T_t f \right),
\]
then
\begin{align}\label{eqn:weaksol}
	\int_{\Omega_{a,b}} \nabla u \cdot \nabla v = 
	-\int_a^b \left(\int_{\pO} (v \circ \varphi_t) \left\{L_t f 
	+ \frac{dg}{dt} + \sigma_t H_t g - T_t g \right\} a_t \,d\mu \right) dt
\end{align}
for any $v \in H^1_0(\Omega_{a,b})$.
\end{lemma}

\begin{proof}
It suffices to consider $v \in C^\infty_0(\Omega_{a,b})$. The coarea formula yields
\begin{align*}
	\int_{\Omega_{a,b}} \nabla u \cdot \nabla v = 
	\int_a^b \left(\int_{\pOt} \sigma (\nabla u \cdot \nabla v) d\mu_t \right) dt.
\end{align*}
On $\pOt$ we use \eqref{eqn:weak} and the definition of $g$ to write
\[
	\left.\nabla u \cdot \nabla v\right|_{\pOt}
	= \nabla^{\pOt} \left(f \circ \varphi_t^{-1}\right) \cdot \nabla^{\pOt} v + \frac{\p v}{\p \nu} \left(g \circ \varphi_t^{-1} \right).
\]
For the tangential part we compute
\begin{align}\label{eqn:tan}
	\int_{\pOt} \sigma \nabla^{\pOt} \left(f \circ \varphi_t^{-1}\right) \cdot \nabla^{\pOt} v \,d\mu_t
	&= -\int_{\pOt} v \dv^{\pOt} \left(\sigma \nabla^{\pOt} \left(f \circ \varphi_t^{-1}\right) \right) d\mu_t 
	\nonumber\\
	&= -\int_{\pO} (v \circ \varphi_t) (L_t f) a_t\, d\mu.
\end{align}
For the normal part we have
\begin{align*}
	\int_{\pOt} \sigma \frac{\p v}{\p \nu} \left(g \circ \varphi_t^{-1} \right) d\mu_t 
	&= \int_{\pO} \sigma_t \left(\frac{\p v}{\p \nu} \circ \varphi_t\right) g a_t \,d\mu \nonumber \\
	&= \int_{\pO} \left[ \frac{d}{dt} (v \circ \varphi_t) - (\gamma \cdot \nabla^{\pOt} v) \circ \varphi_t \right]g a_t\, d\mu.
\end{align*}
The first term on the right-hand side can be written as
\begin{align*}
	\int_{\pO} \frac{d}{dt} (v \circ \varphi_t) g a_t\, d\mu &= \frac{d}{dt} \int_{\pO} (v \circ \varphi_t) g a_t \, d\mu - \int_{\pO} (v \circ \varphi_t) \frac{dg}{dt} a_t \,d\mu - \int_{\pO} (v \circ \varphi_t) g \frac{da_t }{dt} \,d\mu.
\end{align*}
We use the first variation of area formula (Lemma \ref{lem:firstvar}) to obtain
\begin{align*}
	\int_{\pO} (v \circ \varphi_t) g \frac{da_t }{dt} \,d\mu = \int_\pO (v \circ \varphi_t)  \left\{ \sigma_t H_t + \big(\dv^{\pOt} \gamma \big) \circ\varphi_t \right\} g a_t \, d\mu
\end{align*}
and then apply the divergence theorem to the last term to find
\begin{align*}
	\int_\pO \left[\big(v \dv^{\pOt} \gamma \big) \circ\varphi_t \right] g a_t \, d\mu &= \int_\pOt v \dv^{\pOt} \gamma \left(g \circ \varphi_t^{-1}\right) d\mu_t \\
	&= - \int_\pOt \gamma \cdot \left[ v \nabla^{\pOt} \left(g \circ \varphi_t^{-1}\right) + \left(g \circ \varphi_t^{-1}\right) \nabla^{\pOt} v \right] d\mu_t\\
	&= - \int_\pO (v \circ \varphi_t) (T_t g) a_t\, d\mu - \int_\pO \left[(\gamma \cdot \nabla^{\pOt} v) \circ \varphi_t \right] g a_t\, d\mu.
\end{align*}
It follows that
\begin{align}\label{eqn:nor}
	\int_{\pOt} \sigma \frac{\p v}{\p \nu} \left(g \circ \varphi_t^{-1} \right) d\mu_t 
	= \frac{d}{dt} \int_{\pO} (v \circ \varphi_t) g a_t \, d\mu - \int_\pO (v \circ \varphi_t) \left\{ \frac{dg}{dt} + \sigma_t H_t g - T_t g \right\} a_t\, d\mu.
\end{align}
The result follows from adding \eqref{eqn:tan} and \eqref{eqn:nor}, then integrating from $a$ to $b$. The first term from the right-hand side of \eqref{eqn:nor} integrates to zero because $v$ vanishes on $\pO_a$ and $\pO_b$.
\end{proof}

\subsection{Proof of Theorem \ref{thm:local}}\label{sec:localproof}
First assume that $u$ solves \eqref{eqn:PDE} on $\Omega_{0,T}$, in the sense of Definition \ref{def:PDEsol}. This means $F(\cdot,u) \in L^2(\Omega_{a,T})$, and hence $\Delta u \in L^2(\Omega_{a,T})$, for any $a \in (0,T)$. Elliptic regularity (for instance \cite[Theorem 4.16]{M00}) implies that $u \in H^2(\Omega_{a,b})$ for any $0 < a < b < T$, and so $\Tr_t u \in \cH^1 = \Ht \oplus \Hp$ for $t \in (0,T)$. Since $t \mapsto (u \circ \varphi_t)\big|_{\Omega}$ is continuous in $H^2$ for $t \in (0,T)$, we in fact have $\Tr_t u \in C^0\big((0,T),\cH^1\big)$. Next observe that $t \mapsto (u \circ \varphi_t)\big|_{\Omega}$ is differentiable in $H^1$ for $t \in (0,T)$, and continuous for $t \in (0,T]$. Similarly, $t \mapsto \Delta(u \circ \varphi_t)\big|_{\Omega}$ is differentiable in $L^2$ for $t \in (0,T)$ and continuous for $t \in (0,T]$. It follows from \cite[Lemma 3.2]{CJM15} that $\Tr_t u \in C^1\big((0,T),\cH\big) \cap C^0\big((0,T],\cH\big)$. Finally, the coarea formula implies (as in Lemma \ref{lem:L2}) that $F_t(f) \in L^2([a,T],L^2(\pO))$ for any $a \in (0,T)$. Therefore $(f,g) = \Tr_t u$ satisfies the regularity conditions in Definition \ref{def:ODEsol}. We next show that it satisfies the differential equation \eqref{eqn:ODE}.

Taking the normal component of \eqref{eqn:weak}, we obtain
\[
	\left.\frac{\p u}{\p \nu} \right|_{\pOt} = \sigma^{-1} \left( \frac{df}{dt} \circ \varphi_t^{-1} - \gamma \cdot \nabla^{\pOt} \left(f \circ \varphi_t^{-1} \right) \right),
\]
hence
\[
	\frac{df}{dt} = \left(\gamma \cdot \nabla^{\pOt}\left(f \circ \varphi_t^{-1} \right)+ \sigma \frac{\p u}{\p \nu} \right) \circ \varphi_t = T_t f + \sigma_t g.
\]
This verifies the first equation of \eqref{eqn:ODE}.
Next, using \eqref{eqn:weaksol} and the definition of a weak solution to $\Delta u + F(x,u) = 0$, we compute
\begin{align*}
	\int_{\Omega_{a,b}} F(\cdot,u) v &= \int_{\Omega_{a,b}} \nabla u \cdot \nabla v \\
	&= -\int_a^b \left(\int_{\pO} (v \circ \varphi_t) \left\{L_t f 
	+ \frac{dg}{dt} + \sigma_t H_t g - T_t g \right\} a_t \,d\mu \right) dt
\end{align*}
for any $v \in H^1_0(\Omega_{a,b})$. Comparing with
\begin{align*}
	\int_{\Omega_{a,b}} F(\cdot,u)v &= \int_a^b \left(\int_{\pOt} \sigma F(x,u) v\, d\mu \right) dt \\
	&=  \int_a^b \left(\int_{\pO} (v \circ \varphi_t) \sigma_t F_t(f)  a_t \,d\mu \right) dt,
\end{align*}
we find that
\[
	\sigma_t F_t(f) = - \left\{L_t f 
	+ \frac{dg}{dt} + \sigma_t H_t g - T_t g \right\}
\]
which is the second equation of \eqref{eqn:ODE}. The completes the first half of the proof.

Now assume $(f,g)$ satisfies \eqref{eqn:ODE} on $(0,T)$, in the sense of Definition \ref{def:ODEsol}. Define $u$ by \eqref{udef}. We must show that $u$ is a weak solution to \eqref{eqn:PDE} on $\Omega_{0,T}$, i.e. $u$ is a weak solution on $\Omega_{a,b}$ for any $0 < a < b = T$.

For any such $a$ and $b$ we have
\begin{align*}
	f \in C^0([a,b],\Ht) \cap C^0([a,b],\Hp), \\
	g \in C^0([a,b],\Hp) \cap C^0([a,b],\Hm), \\
	F_t(f) \in L^2([a,b], L^2(\pO)).
\end{align*}
In particular, $f \in C^0([a,b],H^1(\pO))$ and $g \in C^0([a,b],L^2(\pO))$, so Lemmas \ref{lem:L2} and \ref{lem:H1} imply $u \in H^1(\Omega_{a,b})$. Moreover, it follows from the coarea formula that $F(\cdot,u) \in L^2(\Omega_{a,b})$.

Let $v \in H^1_0(\Omega_{a,b})$. Using \eqref{eqn:weaksol} and the coarea formula as in the first half of the proof, we obtain
\begin{align*}
	\int_{\Omega_{a,b}} F(\cdot,u)v 	&=  \int_a^b \left(\int_{\pO} (v \circ \varphi_t) \sigma_t F_t(f) \, a_t d\mu \right) dt \\
	&= -\int_a^b \left(\int_{\pO} (v \circ \varphi_t) \left\{L_t f 
	+ \frac{dg}{dt} + \sigma_t H_t g - T_t g \right\} a_t \,d\mu \right) dt \\
	&= \int_{\Omega_{a,b}} \nabla u \cdot \nabla v,
\end{align*}
which says that $u$ is a weak solution to \eqref{eqn:PDE} on $\Omega_{a,b}$. This completes the proof of Theorem \ref{thm:local}.

\subsection{Proof of Theorem \ref{thm:global}}
It is easier to obtain estimates for the $\Hm$ norm of $a_t g(t)$, rather than the norm of $g(t)$ alone. The results obtained below are related to the estimates given in Theorem \ref{thm:global} by the following lemma.

\begin{lemma}
Let $a$ be a positive, continuous function on $\pO$. Then
\[
	(\min a) \|g\|_{\Hm} \leq \|a g \|_{\Hm} \leq (\max a) \|g\|_{\Hm}
\]
for all $g \in \Hm$.
\end{lemma}

\begin{proof}
It suffices to consider smooth $g$. We compute
\begin{align*}
	\|ag\|_{\Hm} &= \sup \left\{ \left| \int_{\pO} agf \,d\mu \right| : f \in C^{\infty}(\pO) \text{ and } \|f\|_{\Hp} = 1 \right\} \\
	&= \sup \left\{ \int_{\pO} agf \,d\mu : f \in C^{\infty}(\pO), \|f\|_{\Hp} = 1 \text{ and } fg \geq 0 \right\} \\
	&\leq \sup \left\{ (\max a) \int_{\pO} gf \,d\mu : f \in C^{\infty}(\pO), \|f\|_{\Hp} = 1 \text{ and } fg \geq 0 \right\} \\
	&= (\max a) \|g\|_{\Hp}.
\end{align*}
Replacing $a$ with $a^{-1}$ and $g$ with $ag$, we obtain
\[
	\|g\|_{\Hm} = \| a^{-1} (ag) \|_{\Hm} \leq (\max a^{-1}) \|ag\|_{\Hm} = \frac{1}{\min a} \|ag\|_{\Hm},
\]
which completes the proof.
\end{proof}

Combining this with Lemma \ref{lem:atbound}, we see that there are constants $c_1$ and $c_2$ such that
\begin{align}\label{gtbound}
	c_1 t^{n-1} \|g\|_{\Hm} \leq \|a_t g \|_{\Hm} \leq c_2 t^{n-1} \|g\|_{\Hm}
\end{align}
for all $g \in \Hm$ and sufficiently small $t>0$.

Keeping \eqref{gtbound} in mind, we begin the proof of Theorem \ref{thm:global}.

First assume that $(f,g)$ is a solution to \eqref{eqn:ODE} on $(0,T)$ satisfying the bound \eqref{fgbound}. Let $u$ be the corresponding weak solution to \eqref{eqn:PDE} on $\Omega_{0,T}$, which exists by Theorem \ref{thm:local}. We must prove that $u \in H^1(\Omega_{T})$, and $u$ is in fact a weak solution on $\Omega_T$. To that end, let $b = T$.

From \eqref{fgbound} we obtain $\|f(t)\|^2_{L^2(\pO)} \leq C t^{-2p}$ for some constant $C$. Computing as in the proof of Lemma~\ref{lem:L2}, and using Lemmas \ref{lem:sigmauniform} and \ref{lem:atbound}, we find
\begin{align*}
	\int_{\Omega_{a,b}} u^2 &= \int_a^b \| \sqrt{\sigma_t a_t} f(t)\|_{L^2(\pO)}^2 dt \\
	& \leq C\int_a^b t^{n-2p-1} \,dt \\
	& \leq C,
\end{align*}
where $C$ does not depend on $a$, since $n-2p-1 > -1$. It follows from the monotone convergence theorem that
\[
	\int_{\Omega_b} u^2 = \lim_{a \to 0^+} \int_{\Omega_{a,b}} u^2 \leq C,
\]
so $u \in L^2(\Omega_b)$.

We next show that $u \in H^1(\Omega_b)$. Since $u$ is a weak solution on $\Omega_{a,b}$ for any $a>0$, Green's first identity implies
\begin{align}\label{Greenab}
	\int_{\Omega_{a,b}} \left( |\nabla u|^2 - uF(\cdot,u) \right) = \int_{\p\Omega_{a,b}} u \frac{\p u}{\p\nu}.
\end{align}
On any $\pOt$ we have
\begin{align*}
	\int_{\pOt} u \frac{\p u}{\p\nu} \,d\mu_t 
	&= \int_{\pO} (u \circ \varphi_t)\left(\frac{\p u}{\p\nu} \circ \varphi_t \right) a_t \,d\mu \\
	&= \int_{\pO} f(t) g(t) a_t \, d\mu
\end{align*}
and so \eqref{fgbound} implies that
\begin{align*}
	\left| \int_{\pOt} u \frac{\p u}{\p\nu} d\mu_t  \right| \leq \|f(t)\|_{\Hp} \| a_t g(t)\|_{\Hm}
\end{align*}
is bounded near $t = 0$. It follows that
\begin{align*}
	\left| \int_{\p\Omega_{a,b}} u \frac{\p u}{\p\nu} \right| &\leq K
\end{align*}
for some constant $K$ that does not depend on $a$. Together with \eqref{Greenab}, this implies
\[
	\|\nabla u\|_{L^2(\Omega_{a,b})}^2 \leq \|u\|_{L^2(\Omega_{a,b})}  \|F(\cdot,u)\|_{L^2(\Omega_{a,b})} + K,
\]
and so $\|\nabla u\|_{L^2(\Omega_{a,b})}^2 \leq C$, with $C$ independent of $a$. The monotone convergence theorem now implies
\[
	\int_{\Omega_b} |\nabla u|^2 = \lim_{a\to0^+} \int_{\Omega_{a,b}} |\nabla u|^2 \leq C
\]
hence $u \in H^1(\Omega_b)$ as was claimed.

Finally, we prove that $u$ is a weak solution to \eqref{eqn:PDE} on $\Omega_b$. Let $v \in C^{\infty}_0(\Omega_b)$. 
Since $u$ is a weak solution on $\Omega_{a,b}$ and $v$ vanishes on $\pO_b$, Green's first identity implies
\begin{align}\label{Greenab2}
	\int_{\Omega_{a,b}} \nabla u \cdot \nabla v = \int_{\Omega_{a,b}} F(\cdot,u) v - \int_{\pO_a} v \frac{\p u}{\p \nu} \,d\mu_a.
\end{align}
Since $\nabla u \cdot \nabla v \in L^1(\Omega_b)$, we can apply the dominated convergence theorem to the functions $(\nabla u \cdot \nabla v)\mathcal{X}_{_{\Omega_{a,b}}}$ to obtain
\[
	\lim_{a\to0^+} \int_{\Omega_{a,b}} \nabla u \cdot \nabla v = \int_{\Omega_b} \nabla u \cdot \nabla v.
\]
It similarly follows that
\[
	\lim_{a\to0^+} \int_{\Omega_{a,b}} F(\cdot,u) v
	= \int_{\Omega_b} F(\cdot,u)v.
\]
The boundary term in \eqref{Greenab2} can be written as\footnote{Here we write $t$ instead of $a$ for the domain $\Omega_{a,b}$, to avoid confusion with the area function $a_t$.}
\[
	\int_{\pO_t} v \frac{\p u}{\p \nu} \,d\mu_t = \int_{\pO} (v \circ \varphi_t) g(t) a_t\, d\mu
\]
and so, using Lemma \ref{lem:Sobolev} and the boundedness of the Sobolev trace map, we obtain
\begin{align*}
	\left| \int_{\pO_t} v \frac{\p u}{\p \nu} \,d\mu_t \right| & \leq \|v \circ \varphi_t\|_{\Hp} \|a_t g(t)\|_{\Hm} \\
	& \leq C \|v \circ \varphi_t\|_{H^1(\Omega)} \|a_t g(t)\|_{\Hm} \\
	& \leq C t^{-n/2} \|v\|_{H^1(\Omega_t)} \|a_t g(t)\|_{\Hm}.
\end{align*}
Since $v$ and $\nabla v$ are bounded, we have $\|v\|_{H^1(\Omega_t)} \leq C t^{n/2}$ for some constant $C$ (which depends on $v$), hence
\[
	t^{-n/2} \|v\|_{H^1(\Omega_t)} \|a_t g(t)\|_{\Hm} \leq C \|a_t g(t)\|_{\Hm} \leq C t^p,
\]
which tends to 0 as $t \to 0$ because $p>0$. Therefore, taking the limit of \eqref{Greenab2} as $a \to 0^+$, we obtain
\[
	\int_{\Omega_b} \nabla u \cdot \nabla v = \int_{\Omega_{b}} F(\cdot,u) v
\]
and so $u$ is a weak solution on $\Omega_b$. This completes the first half of the proof.

Next, assume that $u \in H^1(\Omega_{T})$ is a weak solution to \eqref{eqn:PDE}, and let $(f,g)$ denote the associated solution to \eqref{eqn:ODE} on $(0,T)$, which exists by Theorem \ref{thm:local}. Using Lemma \ref{lem:Sobolev} and the boundedness of the Sobolev trace map $H^1(\Omega) \to \Hp$, we obtain
\begin{align}\label{eqn:fd}
	\|f(t)\|_{\Hp} \leq C \|u \circ \varphi_t\|_{H^1(\Omega)} \leq C t^{-n/2} \|u\|_{H^1(\Omega_t)}.
\end{align}
Elliptic regularity implies $u \in H^2(\Omega_t)$ for any $t < T$, so both $u$ and $\nabla u$ are contained in $H^1(\Omega_t)$.

For $n>2$, the Sobolev embedding theorem implies $u \in L^{2n/(n-2)}(\Omega_t)$, hence $u^2 \in L^{n/(n-2)}(\Omega_t)$. H\"older's inequality then yields
\begin{align*}
	\int_{\Omega_t} u^2 
	& \leq \left\| u^2 \right\|_{ L^{n/(n-2)}(\Omega_t)} \| 1 \|_{L^{n/2}(\Omega_t)} \\
	&= \left( \int_{\Omega_t} |u|^{2n/(n-2)} \right)^{(n-2)/n} |\Omega_t|^{2/n} \\
	&\leq C t^2  \left( \int_{\Omega_t} |u|^{2n/(n-2)} \right)^{(n-2)/n}.
\end{align*}
In the last line we have used the fact that $|\Omega_t| \leq Ct^n$, which can be obtained by choosing $u=1$ in Lemma~\ref{lem:Sobolev}. Similarly estimating the integral of $|\nabla u|^2$ over $\Omega_t$ and then combining with \eqref{eqn:fd}, we see that
\[
	t^{n/2-1} \|f(t)\|_{\Hp} \leq C \left[\left( \int_{\Omega_t} |u|^{2n/(n-2)} \right)^{(n-2)/2n} + \left( \int_{\Omega_t} |\nabla u|^{2n/(n-2)} \right)^{(n-2)/2n}\right].
\]
By the absolute continuity of the Lebesgue integral, we see that the right-hand side tends to 0 as $t \to 0$. This verifies the first term in \eqref{fgdecay}.

For the second term in \eqref{fgdecay}, we let $v \in H^1(\Omega)$, and calculate
\begin{align*}
	\int_\pO g(t) v a_t \, d\mu 
	&= \int_{\pOt} \frac{\p u}{\p \nu} (v \circ \varphi_t^{-1}) \,d\mu_t \\
	&= \int_{\Omega_t} \left[ \nabla u \cdot \nabla (v \circ \varphi_t^{-1}) - F(\cdot,u) (v \circ \varphi_t^{-1}) \right] \\
	&\leq \|\nabla u\|_{L^2(\Omega_t)} \|\nabla (v \circ \varphi_t^{-1})\|_{L^2(\Omega_t)} + 
	\| F(\cdot,u) \|_{L^2(\Omega_t)} \| v \circ \varphi_t^{-1}\|_{L^2(\Omega_t)} \\
	&\leq C \left( t^{n/2-1} \|\nabla u\|_{L^2(\Omega_t)} \|\nabla v \|_{L^2(\Omega)}
	+ t^{n/2} \| F(\cdot,u) \|_{L^2(\Omega_t)} \| v \|_{L^2(\Omega)} \right) \\
	&\leq C \left( t^{n/2-1} \|\nabla u\|_{L^2(\Omega_t)}  + t^{n/2} \| F(\cdot,u) \|_{L^2(\Omega_t)} \right) \| v \|_{H^1(\Omega)}
\end{align*}
where we used the fact that $u$ is a weak solution in the second line, and Lemma \ref{lem:Sobolev} in the penultimate line. It follows that
\begin{align}\label{eqn:gd}
	t^{-n/2} \|a_t g(t)\|_{\Hm} \leq C \left( t^{-1} \|\nabla u\|_{L^2(\Omega_t)} + \| F(\cdot,u) \|_{L^2(\Omega_t)} \right),
\end{align}
so we just need to show that the right-hand side vanishes in the $t=0$ limit.

Applying the Sobolev embedding theorem to $\nabla u \in H^1(\Omega_t)$, as was done for $u$ above, we obtain
\[
	t^{-1} \|\nabla u\|_{L^2(\Omega_t)} \leq C \left( \int_{\Omega_t} |\nabla u|^{2n/(n-2)} \right)^{(n-2)/2n} \to 0.
\]
For the remaining term in \eqref{eqn:gd} we simply observe that $F(\cdot,u) \in L^2(\Omega_t)$ for each $t$, and so
\[
	\int_{\Omega_t} |F(\cdot,u)|^2 \to 0
\]
as $t \to 0$. This establishes \eqref{fgdecay}, and thus completes the proof of Theorem \ref{thm:global} in the case $n>2$.

For the case $n=2$, we return to \eqref{eqn:fd}, with $u \in H^2(\Omega_t)$ for any $t < T$. Now the Sobolev embedding theorem implies $u$ and $\nabla u$ are contained in $L^q(\Omega_t)$ for any $2 \leq q < \infty$; see, for instance \cite[Corollary 9.14]{B11}. We then compute
\begin{align*}
	\int_{\Omega_t} u^2 
	& \leq \left\| u^2 \right\|_{ L^{q/2}(\Omega_t)} \| 1 \|_{L^{q/(q-2)}(\Omega_t)} \\
	&= \left( \int_{\Omega_t} |u|^{q} \right)^{2/q} |\Omega_t|^{(q-2)/q} \\
	&\leq C t^{2(q-2)/q} \left( \int_{\Omega_t} |u|^q \right)^{2/q},
\end{align*}
and similarly for $\nabla u$, to obtain
\[
	t^{2/q} \|f(t)\|_{\Hp} \leq C \left[\left( \int_{\Omega_t} |u|^q \right)^{1/q} + \left( \int_{\Omega_t} |\nabla u|^q \right)^{1/q}\right].
\]
The right-hand side tends to 0 as $t \to 0$, so we obtain the first term in \eqref{fgdecay2D} with $p = \frac2q \in (0,1]$. For the second term in \eqref{fgdecay2D} we use \eqref{eqn:gd} with $n=2$ to obtain
\begin{align*}
	t^{-1} \|a_t g(t)\|_{\Hm} \leq C \left( t^{-1} \|\nabla u\|_{L^2(\Omega_t)} + \| F(\cdot,u) \|_{L^2(\Omega_t)} \right),
\end{align*}
and then observe that
\[
	\int_{\Omega_t} |\nabla u|^2 
	\leq C t^{2(q-2)/q} \left( \int_{\Omega_t} |\nabla u|^q \right)^{2/q},
\]
hence
\[
	t^{2/q-1} \|a_t g(t)\|_{\Hm} \leq C \left[ \left( \int_{\Omega_t} |\nabla u|^q \right)^{1/q} + t^{q/2} \| F(\cdot,u) \|_{L^2(\Omega_t)} \right].
\]
This shows that the second term in \eqref{fgdecay2D} tends to zero for any $p = 1 - \frac2q \in [0,1)$, and thus completes the proof of Theorem \ref{thm:global}.

\section{Exponential dichotomies}\label{sec:dichotomy}
In this final section we discuss exponential dichotomies for the linearization of \eqref{eqn:PDE}. We first define what is meant by an exponential dichotomy for the dynamical system \eqref{eqn:RODE} corresponding to the linearized PDE \eqref{eqn:LPDE}. Next, we explore some consequences of this idea. In particular, we prove that, if a dichotomy exists, then the unstable subspace coincides with the space of Cauchy data for the linear PDE. 

The linear dynamical system \eqref{eqn:RODE} does not satisfy the sufficient conditions given in \cite{PSS97} for the existence of an exponential dichotomy except when the domain is radial, i.e. $\Omega_t = \{x : |x| < t\}$. This case is studied in detail in \cite{BCJLS}, where the existence of an exponential dichotomy is proven. The general case will be the subject of future investigations. For now we simply motivate the concept of an exponential dichotomy by describing some of its consequences for elliptic boundary value problems.

We conclude by giving a dynamical interpretation of an eigenvalue problem, observing that eigenvalues correspond to nontrivial intersections of the unstable subspace with a fixed subspace of $\cH$ that encodes the boundary conditions.

\subsection{Dichotomy subspaces}
Suppose that $\widehat u$ solves \eqref{eqn:PDE}, with the linearized equation $\Delta u + D_u F(x,\widehat u) u = 0$. More generally, consider
\begin{align}\label{eqn:LPDE}
	\Delta u = V(x) u.
\end{align}
The linearized equation, as well as the eigenvalue equation $\Delta u + D_u F(x,\widehat u) u = \lambda u$, can be written in this form. This is a special case of \eqref{eqn:PDE}, with $F(x,u) = -V(x) u$, and hence is equivalent, in the sense of Theorems~\ref{thm:local} and \ref{thm:global}, to the linear system
\begin{align}\label{eqn:LODE}
	\frac{d}{dt} \begin{pmatrix} f \\ g \end{pmatrix} = 
	\begin{pmatrix} T_t & \sigma_t \\ \sigma_t V_t - L_t & T_t - \sigma_t H_t \end{pmatrix}
	\begin{pmatrix} f \\ g \end{pmatrix},
\end{align}
where we have defined $V_t =V \circ \varphi_t \big|_{\pO} \colon \pO \to \bbR$.

The system \eqref{eqn:LODE} is ill-posed, in the sense that solutions do not necessarily exist for given initial data. In \cite{PSS97} it was shown that the corresponding equation \eqref{eqn:cylinder} for the channel problem admits an exponential dichotomy. That is, $\cH = \Hp \oplus \Hm$ splits into two subspaces, both infinite-dimensional, on which the system admits solutions forwards and backwards in time, respectively. However, \eqref{eqn:LODE} does not admit an exponential dichotomy because, as seen in the example in Section~\ref{sec:harmonic}, the solutions decay or grow polynomially, rather than exponentially, in $t$. We will instead consider dichotomies for a suitably reparameterized and rescaled version of the system. 

We let $t = e^\tau$, and then define
\begin{align}\label{scaling}
	\wtf(\tau) = e^{\alpha\tau} f(e^\tau), \quad \wtg(\tau) = e^{(1+\alpha)\tau} g(e^\tau)
\end{align}
for some constant $\alpha$ to be determined. The scaling parameter $\alpha$ will be used to ensure that the asymptotic operator, i.e. the limit of the right-hand side of \eqref{eqn:RODE} as $t \to 0$, does not have spectrum on the imaginary axis. A direct computation shows that if $(f,g)$ solves \eqref{eqn:LODE}, then
\begin{align}\label{eqn:RODE}
	\frac{d}{d\tau} \begin{pmatrix} \wtf \\ \wtg \end{pmatrix} =
	\begin{pmatrix} \alpha + t T_t & \sigma_t \\ t^2 (\sigma_t V_t - L_t) & 1 + \alpha + t (T_t - \sigma_t H_t) \end{pmatrix}
	\begin{pmatrix} \wtf \\ \wtg \end{pmatrix}.
\end{align}
For convenience we set $\wth =(\wtf,\wtg)$.

\begin{definition}\label{thm:dichotomy}
The system \eqref{eqn:RODE} is said to admit an exponential dichotomy on the half line $(-\infty,0]$ if there exists a continuous family of projections $P^u\colon (-\infty,0] \to B(\cH)$ and constants $K, \eta^u, \eta^s > 0$ such that, for every $\tau_0 \leq 0$ and $z \in \cH$ there exists a solution $\wth^u(\tau;\tau_0,z)$ of \eqref{eqn:RODE}, defined for $\tau \leq \tau_0$, such that
\begin{itemize}
	\item $\wth^u(\tau_0;\tau_0,z) = P^u(\tau_0) z$,
	\item $\| \wth^u(\tau; \tau_0,z) \|_{\cH} \leq K e^{\eta^u(\tau-\tau_0)} \|z\|_\cH$ for all  $\tau \leq \tau_0$,
	\item $\wth^u(\tau; \tau_0,z) \in R(P^u(\tau))$ for all $\tau \leq \tau_0$,
\end{itemize}
and a solution $\wth^s(\tau;\tau_0,z)$ of \eqref{eqn:RODE}, defined for $\tau_0 \leq \tau \leq 0$, such that
\begin{itemize}
	\item $\wth^s(\tau_0;\tau_0,z) = P^s(\tau_0) z$,
	\item $\| \wth^s(\tau; \tau_0,z) \|_{\cH} \leq K e^{\eta^s(\tau_0-\tau)} \|z\|_\cH$ for all $\tau_0 \leq \tau \leq 0$,
	\item $\wth^s(\tau; \tau_0,z) \in R(P^s(\tau))$ for all $\tau_0 \leq \tau \leq 0$,
\end{itemize}
where $P^s(\tau) = I - P^u(\tau)$.
\end{definition}

In other words, for any terminal data in the range of $P^u(\tau_0)$, the system can be solved backwards in $\tau$, with the solution decaying exponentially as $\tau \to -\infty$, and similarly for initial data in the range of $P^s(\tau_0)$.

For any $\tau \leq 0$ we define the stable and unstable subspaces
\begin{align}
	\widetilde E^s(\tau) = R(P^s(\tau)), \quad 	\widetilde E^u(\tau) = R(P^u(\tau)).
\end{align}
Undoing the scaling \eqref{scaling} and the change of variables $t = e^\tau$, we define
\begin{align}
	E^u(t) = \left\{ \left( t^{-\alpha} \wtf(\log t), t^{-1 - \alpha} \wtg(\log t) \right) : \big(\wtf(\log t), \wtg(\log t) \big) \in \widetilde E^u(\log t) \right\}
\end{align}
for $t>0$, and similarly for $E^s(t)$. Thus for any $t_0 \in (0,1]$ and $(f_0,g_0) \in E^u(t_0)$ there exists a solution $(f(t), g(t))$ to \eqref{eqn:LODE}, defined for $0 < t \leq t_0$, with $(f(t_0), g(t_0)) = (f_0, g_0)$ and
\begin{align}\label{unstabledecay}
	\big\| (f(t), t g(t)) \big\|_{\cH} \leq K \left(\frac{t}{t_0}\right)^{\eta^u-\alpha} \big\| (f_0, t_0 g_0) \big\|_{\cH}.
\end{align}
The implications of this estimate for the corresponding solution $u$ to the linear PDE \eqref{eqn:LPDE} depend on the scaling parameter $\alpha$ and its relation to the growth and decay rates $\eta^u$ and $\eta^s$. This will be explored in detail in the following section.

\subsection{Relation to the Cauchy data space}
Assuming the existence of an exponential dichotomy, we now prove that for an appropriate choice of $\alpha$ the unstable subspace $E^u(t)$ corresponds to the space of Cauchy data of weak solutions to \eqref{eqn:LPDE} on $\Omega_t$. For $t>0$ let
\[
	K_t = \{u \in H^1(\Omega_t) : \Delta u = V(x) u \text{ on } \Omega_t \},
\]
where the equality $\Delta u = V(x) u$ on $\Omega_t$ is meant in a distributional sense. Since $K_t$ is a subset of $\{u \in H^1(\Omega_t) : \Delta u \in L^2(\Omega_t)\}$, the trace map $\Tr_t$ (defined in \eqref{def:tr}) can be applied, and we have $\Tr_t u \in \Hh$ for each $u \in K_t$. We thus define
\[
	\Tr_t(K_t) = \{\Tr_t u : u \in K_t\} \subset \cH.
\]

\begin{theorem}\label{thm:unstable}
Assume that \eqref{eqn:RODE} admits an exponential dichotomy on $(-\infty,0]$, and that $V$ is of class $C^{\lfloor n/2 \rfloor,1}$ in a neighborhood of the origin. If $\alpha < \eta^u + \frac{n}{2} - 1$, then $E^u(t) \subseteq \Tr_t(K_t)$ for each $t > 0$. If $\alpha > - \eta^s$, then $ \Tr_t(K_t) \subseteq E^u(t)$ for each $t > 0$. Therefore, the two spaces coincide if
\begin{align}\label{alpharange}
	-\eta^s < \alpha < \eta^u + \frac{n}{2} - 1.
\end{align}
\end{theorem}

To prove that $E^u(t) \subseteq \Tr_t(K_t)$, we must show that any solution $(f,g)$ to \eqref{eqn:LODE} having sufficient decay at $t=0$ corresponds to a solution $u$ to \eqref{eqn:LPDE}. Thus suppose $(f(t_0),g(t_0)) \in E^u(t_0)$ for some $t_0 > 0$. It follows from \eqref{unstabledecay} that
\[
	\|f(t)\|_{\Hp} + t \|g(t)\|_{\Hm} \leq C t^{\eta^u-\alpha}
\]
for all $t \leq t_0$. The inequality $\alpha < \frac{n}{2} - 1 + \eta^u$ implies $\alpha - \eta^u \leq n-2 + \eta^u - \alpha$, $\alpha - \eta^u < \frac{n}{2}$ and $n-2 + \eta^u - \alpha > 0$. Therefore, there exists a number
\[
	p \in \left(0, \frac{n}{2}\right) \cap [\alpha - \eta^u, n-2 + \eta^u - \alpha].
\]
For this choice of $p$ we have that
\[
	t^p \|f(t)\|_{\Hp} + t^{n - p- 1} \|g(t)\|_{\Hm}
\]
is bounded near $t=0$. It follows from Theorem \ref{thm:global} that $(f(t),g(t)) = \Tr_t u$ for some $u \in K_t$, completing the proof that $E^u(t_0) \subseteq \Tr_{t_0}(K_{t_0})$.

To prove the reverse inclusion, $ \Tr_t(K_t) \subseteq E^u(t)$, we use the fact that any solution to \eqref{eqn:LODE} having sufficient decay at $t=0$ is necessarily contained in the unstable subspace.

\begin{lemma}\label{inclusion}
Suppose $\wth$ is a solution to \eqref{eqn:RODE} and satisfies the estimate $\|\wth(\tau)\| \leq C e^{\alpha(\tau - \tau_0)}$ for all $\tau \leq \tau_0$, for some $\alpha > -\eta^s$. Then $\wth(\tau_0) \in \widetilde E^u(\tau_0)$.
\end{lemma}

That is, any solution $\wth(\tau)$ that does not blow up too rapidly as $\tau \to -\infty$ must be contained in the unstable subspace, so it in fact decays with rate $\eta^u$.

\begin{proof}
Choosing  $z = h(\tau_0)$ in Definition \ref{thm:dichotomy}, there exists a solution $\wth^u$ to \eqref{eqn:RODE} with $\wth^u(\tau_0) = P^u(\tau_0) z$, satisfying the estimate $\|\wth^u(\tau)\| \leq K e^{\eta^u(\tau - \tau_0)}\|z\|$ for $\tau \leq \tau_0$. Now define $\wth^s(\tau) = h(\tau) - \wth^u(\tau)$. It follows that $\wth^s$ is also a solution to \eqref{eqn:RODE} for $\tau \leq \tau_0$, with
\begin{align}\label{hest1}
	\|\wth^s(\tau)\| \leq C e^{\alpha(\tau - \tau_0)} + K \|z\| e^{\eta^u(\tau - \tau_0)}.
\end{align}
Moreover, since $\wth^u(\tau) \in R(P^u(\tau))$, we have $\wth^s(\tau) = (I - P^u(\tau)) \wth(\tau) \in R(P^s(\tau))$.
Now let $\tau_* < \tau_0$. Since $\wth^s(\tau)$ solves \eqref{eqn:RODE} for $\tau \geq \tau_*$, and has initial condition $\wth^s(\tau_*)$, it must be the unique forward-in-time solution whose existence is guaranteed by the exponential dichotomy. Therefore it satisfies the estimate
\[
	\|\wth^s(\tau)\| \leq K e^{\eta^s(\tau_* - \tau)} \|\wth^s(\tau_*)\| 
\]
for $\tau \geq \tau_*$. Using \eqref{hest1} to bound $\|\wth^s(\tau_*)\|$, we have
\[
	\|\wth^s(\tau_0)\| \leq K e^{\eta^s(\tau_* - \tau_0)}  \left(C e^{\alpha(\tau_* - \tau_0)} + K \|z\| e^{\eta^u(\tau_* - \tau_0)}\right).
\]
Taking the limit $\tau_* \to -\infty$ and using the fact that $\eta^s + \alpha > 0$, we obtain $\wth^s(\tau_0) = 0$.
\end{proof}

We also require the following improved version of Theorem \ref{thm:global}.

\begin{lemma}\label{lemma:lineardecay}
Suppose $u$ solves \eqref{eqn:LPDE} on $\Omega_T$, with the potential $V$ of class $C^{\lfloor n/2 \rfloor,1}$ in a neighborhood of the origin. Then $\|f(t)\|_{\Hp}$ and $\|g(t)\|_{\Hm}$ are bounded near $t=0$.
\end{lemma}

\begin{proof}
The smoothness of $V$ allows us to use elliptic regularity (for instance \cite[Theorem 4.16]{M00}) to conclude that $u \in H^{2 + \lfloor n/2 \rfloor }(\Omega_t)$ for sufficiently small $t$. It follows from the Sobolev inequalities that $u \in C^{1,\gamma}(\Omega_t)$ for some $\gamma \in (0,1)$. In particular, $u$ and $\nabla u$ are uniformly bounded in a neighborhood of the origin, and so $\| u \|_{H^1(\Omega_t)} \leq C t^{n/2}$. The result now follows from estimates \eqref{eqn:fd} and \eqref{eqn:gd}.
\end{proof}

Now consider $\Tr_t u \in \Tr_t (K_t)$. We have that $\|f(t)\|_{\Hp}$ and $\|g(t)\|_{\Hm}$ are bounded as $t \to 0$. Therefore $\wth(\tau) = \big(e^{\alpha\tau} f(e^\tau), e^{(1+\alpha)\tau} g(e^\tau)\big)$ satisfies the bound $\|\wth(\tau)\|_\cH \leq C e^{\alpha\tau}$. Since $\alpha > -\eta^s$, Lemma \ref{inclusion} implies $\wth(\tau) \in \widetilde E^u(\tau)$, hence $h(t) = (f(t),g(t)) \in E^u(t)$. This completes the proof of Theorem \ref{thm:unstable}.

Note that the rates $\eta^u$ and $\eta^s$ depend implicitly on $\alpha$, as the latter parameter appears in the rescaled system of equations \eqref{eqn:RODE}. To verify \eqref{alpharange} one must therefore understand this dependence.

\begin{cor}\label{cor:alphagap}
Suppose $n>2$. If $0 < \alpha \leq \frac{n}{2} - 1$ and \eqref{eqn:RODE} has an exponential dichotomy with rates $\eta^{s,u} > 0$, then \eqref{alpharange} is satisfied, and hence $E^u(t) = \Tr_t(K_t)$ for each $t > 0$.
\end{cor}

\begin{rem}\label{rem:n2}
When $n=2$ there is no $\alpha$ that satisfies this condition. This is not a shortcoming of the method of proof, but rather indicates a fundamental difference between the cases $n=2$ and $n>2$. This was seen earlier when studying harmonic functions on the plane. As observed above in Remark \ref{rem:2Dharmonic}, there is no choice of $\alpha$ for which $E^u(t) = \Tr_t(K_t)$.
\end{rem}

\subsection{Application to an eigenvalue problem}

Finally, we use Corollary \ref{cor:alphagap} to give a dynamical interpretation of the eigenvalue problem
\begin{align}\label{eigenvalue}
	-\Delta u + V u = \lambda u
\end{align}
with Dirichlet boundary conditions. To do so we define the \emph{Dirichlet subspace}
\begin{align}
	\cD = \{( 0,g) : g \in \Hm\} \subset \cH.
\end{align}

The following result is then an immediate consequence of Corollary \ref{cor:alphagap}.

\begin{cor}
Assuming the hypotheses of Corollary \ref{cor:alphagap}, $\lambda$ is an eigenvalue of the Dirichlet problem \eqref{eigenvalue} on $\Omega_t$ if and only if the unstable subspace $E^u(t)$ intersects the Dirichlet subspace $\cD$ nontrivially. Moreover, the geometric multiplicity of $\lambda$ equals $\dim \big(E^u(t) \cap \cD\big)$.
\end{cor}

Other boundary conditions (Neumann, Robin, etc.) can be characterized in a similar way by changing $\cD$ accordingly; see \cite{CJLS16,CJM15} for details.

Our construction thus gives a dynamical perspective on elliptic eigenvalue problems, similar to the Evans function \cite{KPP13}, which counts intersections between stable and unstable subspaces. (While traditionally developed for problems in one spatial dimension, some progress has been made on extending the Evans function to channel domains; see \cite{DN06,DN08,GLZ08,LP15,OS10}.)

This is also closely related to the Maslov index, a symplectic winding number that counts intersections of Lagrangian subspaces in a symplectic Hilbert space; see \cite{CJLS16,CJM15,DJ11,LS18}.

\section*{References}

\bibliography{MaslovHam}

\end{document}